\documentclass[final]{dmtcs-episciences}
\usepackage{amsmath,amssymb,amsfonts}
\usepackage[english]{babel}
\usepackage{hhline}
\usepackage{amsthm}
\usepackage{latexsym}
\usepackage{graphicx}
\usepackage{color}
\usepackage[table]{xcolor}
\usepackage{diagbox}
\usepackage{lineno}
\usepackage{url}

\newcommand\blfootnote[1]{%
  \begingroup
  \renewcommand\thefootnote{}\footnote{#1}%
  \addtocounter{footnote}{-1}%
  \endgroup
}

\newtheorem{theorem}{Theorem}

\newtheorem{corollary}[theorem]{Corollary}

\newtheorem{lemma}[theorem]{Lemma}

\newtheorem{openproblem}{Open Problem}
\newtheorem{observation}{Observation}

\author{
  B. M. \'Abrego\affiliationmark{1}
  \and
  S. Fern\'andez-Merchant\affiliationmark{1}\thanks{Supported by the NSF grant DMS-1400653.}
  \and
  M. Kano\affiliationmark{2}\thanks{Supported by JSPS KAKENHI Grant Number 16K05248.}\\
  \and
  D. Orden\affiliationmark{3}\thanks{Supported by Project MTM2017-83750-P of the Spanish Ministry of Science (AEI/FEDER, UE).}
  \and
  P. P\'erez-Lantero\affiliationmark{4}\thanks{Supported by CONICYT FONDECYT/Regular 1160543 (Chile) and Millennium Nucleus Information and Coordination in Networks ICM/FIC RC130003 (Chile).}
  \and
  C. Seara\affiliationmark{5}\thanks{Supported by the projects Gen. Cat. DGR 2014SGR46 and MINECO MTM2015-63791-R.}
  \and
  J. Tejel\affiliationmark{6}\thanks{Supported by MINECO project MTM2015-63791-R and Gobierno de Arag\'on under Grants E58 (ESF) and E41-17R.}
}

\title{$K_{1,3}$-covering red and blue points in the plane\thanks{An extended abstract of this work has appeared at the 33rd European Workshop on Computational Geometry (EuroCG 2017).}}

\affiliation{
California State University, Northridge, USA\\
Ibaraki University, Japan\\
Universidad de Alcal\'{a}, Spain\\
Universidad de Santiago, Chile\\
Universitat Polit\`{e}cnica de Catalunya, Spain\\
Universidad de Zaragoza, Spain
}


\received{2018-5-25}

\revised{2019-1-15}

\accepted{2019-1-16}

\begin{document}
\publicationdetails{21}{2019}{3}{6}{4537}
\maketitle

\blfootnote{\begin{minipage}[l]{0.3\textwidth} \includegraphics[trim=10cm 6cm 10cm 5cm,clip,scale=0.15]{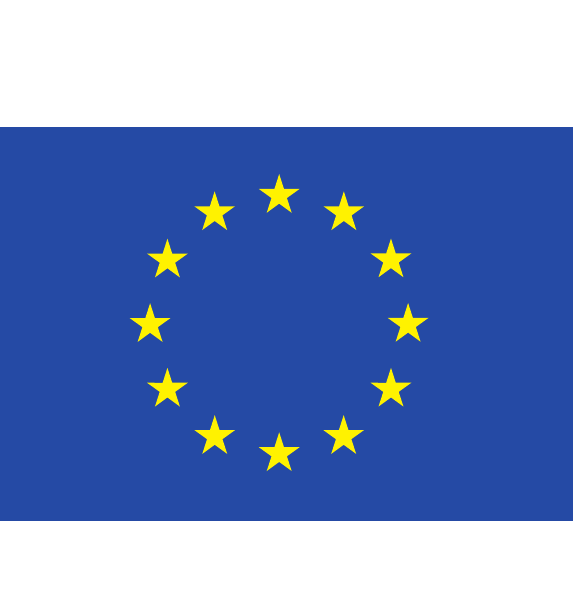} \end{minipage}  \hspace{-2cm} \begin{minipage}[l][1cm]{0.7\textwidth}
 	  This work has received funding from the European Union's Horizon 2020 research and innovation programme under the Marie Sk\l{}odowska-Curie grant agreement No 734922.
 	\end{minipage}}
 	
\begin{abstract}
We say that a finite set of red and blue points in the plane in general position can be $K_{1,3}$-covered if the set can be partitioned into subsets of size $4$, with $3$ points of one color and $1$ point of the other color, in such a way that, if at each subset the fourth point is connected by straight-line segments to the same-colored points, then the resulting set of all segments has no crossings.
We consider the following problem: \emph{Given a set $R$ of $r$ red points and a set $B$ of $b$ blue points in the plane in general position, how many points of $R\cup B$ can be $K_{1,3}$-covered}? and we prove the following results:
\begin{itemize}
\item[(1)] If $r=3g+h$ and $b=3h+g$, for some non-negative integers $g$ and $h$, then there are point sets $R\cup B$, like $\{1,3\}$-equitable sets (i.e., $r=3b$ or $b=3r$) and linearly separable sets, that can be $K_{1,3}$-covered.
\item[(2)] If $r=3g+h$, $b=3h+g$ and the points in $R\cup B$ are in convex position, then at least $r+b-4$ points can be $K_{1,3}$-covered, and this bound is tight.
\item[(3)] There are arbitrarily large point sets $R\cup B$ in general position, with $r=b+1$, such that at most $r+b-5$ points can be $K_{1,3}$-covered.
\item[(4)] If $b\le r\le 3b$, then at least $\frac{8}{9}(r+b-8)$ points of $R\cup B$ can be $K_{1,3}$-covered. For $r>3b$, there are too many red points and at least $r-3b$ of them will remain uncovered in any $K_{1,3}$-covering.
\end{itemize}
Furthermore, in all the cases we provide efficient algorithms to compute the corresponding coverings.
\end{abstract}

\section{Introduction}\label{sec-1}

\subsection{Description of the problem}\label{sect:background}

Given sets $R$ of red points and $B$ of blue points in the plane, such that $R\cup B$ is in general position (there are no three collinear points), we say that a graph $G$ \emph{covers} $R\cup B$ if: (i)~The vertex set of $G$ is $R\cup B$, (ii)~every edge of $G$ is a straight line segment connecting a red point and a blue point, and (iii)~no two edges intersect except in their endpoints. Analogously, for a fixed graph $G$ of $k$ vertices, we say that $R\cup B$ has a $G$-\emph{covering}, or can be $G$-\emph{covered}, if $|R\cup B|=t\cdot k$ for some integer $t$, and the graph $G_t$ resulting from the union of $t$ copies of~$G$ covers $R\cup B$. (Note the difference between this notion, of covering a point set with copies of a graph, and the classical vertex cover problem in graphs~\cite{Angel2018}).

Let us consider as graph~$G$ the complete bipartite graph $K_{1,n}$, which we will call the \emph{star} of order~$n+1$ \emph{centered} at the partition of size~$1$.
It is well known that, for $S=R\cup B$ in general position with $|R|=|B|$, a $K_{1,1}$-covering always exists: Using recursively the Ham-Sandwich theorem allows to find a \emph{non-crossing geometric alternating perfect matching}~\cite{L1983}. Furthermore, such a matching can be extended, adding edges, to a non-crossing geometric alternating spanning tree~\cite{hurtado2008}.

It is also known that, for $S=R\cup B$ in general position with $|R|=2g+h$ and $|B|=g+2h$, a $K_{1,2}$-covering always exists. Moreover, this covering uses $g$ stars centered at a blue point and $h$ stars centered at a red point. The result follows from this theorem

\begin{theorem}[Kaneko, Kano, and Suzuki~\cite{KKS2004}]\label{th:A}
Let $g$ and $h$ be non-negative integers. If $n$ is an even integer such that $2\le n \le 12$, then any set of $(n/2)g$ red points and $(n/2)g$ blue points in the plane in general position can be $P_n$-covered. If $n$ is an odd integer such that $3\le n\le 11$, then any set of $\lceil n/2\rceil g+\lfloor n/2 \rfloor h$  red points and $\lfloor n/2 \rfloor g+\lceil n/2\rceil h$ blue points in the plane in general position can be $P_n$-covered.
\end{theorem}

\noindent which also provides a sufficient condition for a given set of red and blue points to have a $P_n$-covering, being $P_n$ the (non-crossing alternating) path of length~$2\leq n\leq 12$. These bounds are the best possible, since when $n=13$ or $n\ge 15$, there exist configurations of $\lceil n/2\rceil$ red points and $\lfloor n/2 \rfloor$ blue points for which there does not exist any $P_n$-covering~\cite{KKS2004}.
In addition, one can build point configurations with $m$ red points and $m$ blue points in convex position such that the longest non-crossing alternating path can cover at most $\frac{4}{3}m+o(m)$ of the $2m$ points~\cite{AGHT2003,KPT2008}. Therefore, for $n$ large in relation to the total number of red and blue points, a big portion of the points could be left uncovered when trying to cover the points with paths of size $n$.

\begin{figure}[ht]
\begin{center}
\includegraphics[width=\textwidth]{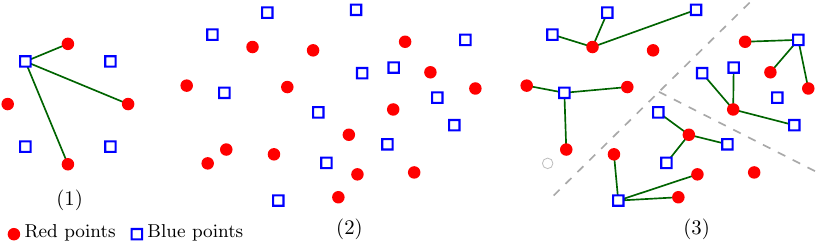}
\caption{(1) At most $4$ points of given $8$ points can be $K_{1,3}$-covered. (2) A set of $15$ red points and $13$ blue points. (3) There is a subdivision of the plane into three convex regions that induces a partition of the $14$ red points and $13$ blue points into three sets containing either $5$ red and $4$ blue points, or $4$ red and $5$ blue points; inside each set, $8$ points can be $K_{1,3}$-covered. Note that point set in (3) is obtained from (2) by removing one red point.}\label{fig:1}
\end{center}
\end{figure}

Knowing that $K_{1,1}$- and $K_{1,2}$-coverings do always exist, this paper goes one step further and considers the problem for $K_{1,3}$-coverings. It turns out that such coverings are not always possible: In Section~\ref{sec-2} we exhibit some bicolored point sets which do not admit a $K_{1,3}$-covering, although we also show that there exist bicolored point sets which can always be $K_{1,3}$-covered. Hence, we study the problem of \emph{given a set of $R$ red points and $B$ blue points in the plane in general position, how many points of $R\cup B$ can be $K_{1,3}$-covered}?  See Figure~\ref{fig:1}. In Section~\ref{sec-3}, we prove that, for any set $R$ of red points and any set $B$ of blue points in the plane in general position such that $|B|\le |R|\le 3|B|$, at least $\frac{8}{9}(|R|+|B|-8)$ points can be $K_{1,3}$-covered. Furthermore, for the configurations studied in these sections, we show efficient algorithms for computing the corresponding coverings. Finally, in Section~\ref{sec-4} we detail some concluding remarks and open problems.

\subsection{Related work}\label{sect:related}

A considerable amount of research about discrete geometry on red and blue points in the plane has been done. Questions about how to cover the points using specific geometric structures, or how to partition the plane into convex pieces such that each piece contains roughly the same number of red and blue points, have been widely studied in the literature. The reader is referred to~\cite{KK2003} (and the references therein) for a survey on this topic.

One of the more challenging problems on covering red and blue points, proposed by Erd\H{o}s~\cite{PH1989} in 1989 (see also \cite{AU90}), is that of determining the largest number $s(n)$ such that, for every set of $n$ red and $n$ blue points on a circle, there exists a \emph{non-crossing geometric alternating path} consisting of $s(n)$ vertices. Erd\H{o}s conjectured that $s(n)= \frac{3}{2}n+2+o(n)$, but one can find point configurations for which  $s(n)<\frac{4}{3}n +o(n)$, as described in~\cite{AGHT2003,KPT2008}. The best bounds up to date for $s(n)$ are due to Kyn\v{c}l, Pach and T\'{o}th~\cite{KPT2008}, and valid for bicolored point sets in general position. However, the conjecture that $|s(n)-\frac{4}{3}n|= o(n)$ remains open, even for points in convex position.

If crossings are allowed, Kaneko et al.~\cite{KKY2000} proved that, for $R$ and $B$ in general position with $|R|=|B|$, there exists a geometric Hamiltonian alternating cycle on $R\cup B$ that has at most $|R|-1$ crossings. Moreover, this bound is best possible. Claverol et al.~\cite{claverol2016} obtain the same result, but requiring the geometric Hamiltonian alternating cycle to be 1-plane (that is, every edge is crossed at most once). They also give an upper bound on the number of crossings of such a cycle, which depends on the number of \emph{runs} of the points of $S=R\cup B$ on the boundary of the convex hull of $S$ (a run is a maximal set of consecutive points of the same color). Garc\'{\i}a and Tejel~\cite{GT2017} give a polynomial algorithm to find the \emph{shortest} Hamiltonian alternating cycle for a bicolored bipartite graph $G(R\cup B,E)$ satisfying the quadrangle property, as turns out to be the case for a bicolored point set $R\cup B$ in the plane in convex position.

While trying to cover a bicolored point set in the plane with a non-crossing geometric Hamiltonian alternating path or cycle is not always possible, this is not the case for \emph{non-crossing geometric alternating spanning trees}. One only needs to take any red point and connect it to all the blue points. Then, extending the resulting edges from the blue endpoints to partitioning the plane into cones, the remaining red points in each cone can be connected to a suitable blue point on the boundary of that cone. One can even bound the maximum vertex degree: Abellanas et al.~\cite{AGHNR1999} proved that, if $|R|=|B|$, then there exists a non-crossing geometric alternating spanning tree on $R\cup B$ having maximum degree at most $O(\log(|R|))$, and Kaneko~\cite{K2000} proved a bound of at most~3. Recently, Biniaz et al.~\cite{biniaz2016} proved the existence of a non-crossing geometric alternating spanning tree of degree at most $\max\{3, \lceil\frac{|R|-1}{|B|} \rceil +1\}$. As a counterpart to covering problems with alternating graphs, in~\cite{DK2001,T1996} covering problems with monochromatic graphs (paths, matchings or trees) are studied. We also refer the reader to~\cite{bereg2012,chan2015,evans2016,fulek2013,giacomo2012,hurtado2009} for other related problems about covering red and blue points in the plane and to~\cite{KSU2014} for coverings involving points colored with more than two colors.

A related family of problems is that of balanced subdivision problems, in which one would like to partitioning the plane (or $\mathbb{R}^d$) into convex regions in such a way that some desired properties on the number of red and blue points in each region are achieved. Perhaps, the most celebrated result in this family of problems is the Ham-Sandwich theorem which states that, for any $d$ point sets in $\mathbb{R}^d$ (in general, $d$ measures), there is a hyperplane bisecting each of these sets (measures). Apart from the intrinsic theoretical interest of these partitioning problems, the relevance of many of the results obtained for these problems relies on the fact that they have a lot of applications. For instance, some of the results given in Theorem \ref{th:A} are based on particular partitions of the plane. As we explain in the forthcoming sections, we also apply some of these partitions to obtain some of our results. For a wide range of geometric partitioning results, the reader is referred to~\cite{BKS2000,HKV2016,Ito2000,KK1999,KK2003,KKS2004,KK2016,Sakai2002} and the references therein.

\subsection{Some notation and useful lemmas} \label{sect:lemmas}

Throughout this paper, we will always assume that no three points are collinear. We denote by $R$ a set of red points in the plane and by $B$ a set of blue points in the plane. The convex hull of any point set $S$ is denoted by $conv(S)$ and a finite point set $S$ in the plane is in {\it convex position} if the points of $S$ are the vertices of $conv(S)$. Abusing the notation, when no confusion can arise we will also use $conv(S)$ to denote the boundary of the convex hull. Throughout the paper, we deal with \emph{directed lines} in order to define the right side of a line and the left side of it. Thus, a \emph{line} will mean a directed line. A line $\ell$ dissects the plane into three pieces: $\ell$ and two open half-planes $left(\ell)$ and $right(\ell)$, where $left(\ell)$ and $right(\ell)$ denote the \emph{open half-planes} to the left of $\ell$ and to the right of~$\ell$, respectively. Notice that, if $\ell^*$ denotes a line on $\ell$ with the opposite direction of $\ell$, then $left(\ell^*)=right(\ell)$ and $right(\ell^*)=left(\ell)$.

For a set $S=R\cup B$ of red and blue points in general position, if $R$ consists of exactly $r$ red points and $B$ consists of exactly $b$ blue points, we say that $S$ is an \emph{$(r,b)$-set}. We define $\mathcal{C}(S)$ as the maximum number of points in $S$ that can be $K_{1,3}$-covered. For non-negative integers $r$ and $b$, we define $\mathcal{C}(r,b)$ as the minimum of $\mathcal{C}(S)$ over all $(r,b)$-sets $S$ in general position. Sometimes it is better to look at the number of points that are left \emph{uncovered}. Hence, we define $\mathcal{U}(S)=|S|-\mathcal{C}(S)$ and $\mathcal{U}(r,b)=r+b-\mathcal{C}(r,b)$. Then, $\mathcal{U}(r,b)$ is the maximum of $\mathcal{U}(S)$ over all $(r,b)$-sets $S$ in general position.

To finish this section, we present two lemmas which might be useful also in a more general context. They are similar to other well-known results of this type and their proofs are based on a simple continuity argument. The first one is an intermediate value result on red and blue points.

\begin{lemma}\label{lem:9}
Let $S = R\cup B$ be a bicolored point set in the plane. If there exist two directed lines $\ell_1$ and $\ell_2$ such that $|left(\ell_1)\cap (R\cup B)|=|left(\ell_2)\cap (R\cup B)|=m$ and $|left(\ell_1)\cap B|<|left(\ell_2)\cap B|$, then:
\begin{itemize}
  \item[(i)] For every integer $j$, $|left(\ell_1)\cap B|\leq j \leq |left(\ell_2)\cap B|$, there exists a directed line~$\ell_3$ such that $|left(\ell_3)\cap (R\cup B)|=m$ and $|left(\ell_3)\cap B|=j$.
  \item[(ii)]\label{lem:9(ii)} There is a directed line $\ell_4$ through a blue point such that $|left(\ell_4)\cap (R\cup B)|=m-1$ and $|left(\ell_4)\cap B|=|left(\ell_1)\cap B| $.
\end{itemize}
 Moreover, $\ell_3$ and $\ell_4$ can be found in $O(N^{\frac{4}{3}}\log (N))$ time, where $N = |S|$.
\end{lemma}

\begin{proof}
For each direction $\theta$, let $\ell_{\theta}(m)$ be a directed line with  direction~$\theta$ such that $|left(\ell_{\theta}(m))\cap (R\cup B)| =m$, if it exists. The line $\ell_{\theta}(m)$ does not exist precisely when there is a line $\ell'_{\theta}(m)=\overrightarrow{xy}$ in direction $\theta$ through two points $x$ and $y$ in $R\cup B$ such that $|left(\ell'_{\theta}(m))\cap (R\cup B)|=m-1$. Let $S_{\theta}(m)=left(\ell_{\theta}(m))\cap (R\cup B)$ if $\ell_\theta(m)$ exists, and $S_{\theta}(m)=\{x\}\cup (left(\ell'_{\theta}(m))\cap (R\cup B))$ otherwise.

As $\theta$ changes continuously from $0$ to $2\pi$ (counterclockwise), the set $S_\theta$ changes finitely many times, precisely each time that $\ell_\theta (m)$ does not exist. So every time $S_\theta$ changes, it does so only by one point. Namely, if the change happens at $\ell'_\theta (m)=\overrightarrow{xy}$, then for $\varepsilon>0$ small enough $S_{\theta}(m)=S_{\theta-\varepsilon}(m)\cup \{x\}\setminus\{y\}$. At that point, $|S_{\theta}(m)\cap B|$ and $|S_{\theta}(m)\cap R|$ do not change when $x$ and $y$ have the same color; $|S_{\theta}(m)\cap B|$ increases by one and $|S_{\theta}(m)\cap R|$ decreases by one when $x$ is blue and $y$ is red; and $|S_{\theta}(m)\cap B|$ decreases by one and $|S_{\theta}(m)\cap R|$ increases by one when $x$ is red and $y$ is blue. So $|S_{\theta}(m)\cap B|$ achieves all possible values between $|left(\ell_1)\cap B|$ and $|left(\ell_2)\cap B|$.

In particular, if $\ell'_\theta(m)=\overrightarrow{xy}$ and at this moment $|S_{\theta}(m)\cap B|$ is decreasing from $|left(\ell_1)\cap B|+1$ to $|left(\ell_1)\cap B|$, then $x$ is red, $y$ is blue, and for some $\varepsilon>0$ small enough, the line $\ell_4$ in direction $\theta-\varepsilon$ passing through $y$ satisfies that $|left(\ell_4)\cap (R\cup B)|=m-1$ and $|left(\ell_4)\cap B|=|left(\ell_1)\cap B|$.

Observe that, in the rotation process, we are building a subset of the $m$-sets of $S$. For $n(k)$ denoting the number of $k$-sets of a point set $S$, Edelsbrunner and Welzl~\cite{EW1986} gave an $O(N\log(N)+n(k)\log^2 (N))$ algorithm to find all $k$-sets. Their algorithm is based on the dynamic algorithm by Overmars and van Leeuwen \cite{OL1981} that maintains a convex hull in $O(\log^2 (N))$ time per update. Using the algorithm given by Brodal and Jacob in~\cite{BJ02}, that maintains a convex hull in $O(\log (N))$ time per update, the algorithm from Edelsbrunner and Welzl can be directly improved to $O(N\log(N)+n(k)\log (N))$ time.
As $n(k)=O(Nk^{\frac{1}{3}})$~\cite{D1998}, lines $\ell_3$ and $\ell_4$ can be found in $O(N^{\frac{4}{3}}\log (N))$ time.
\end{proof}

\begin{lemma}\label{lem:10}
Let $S$ be a point set in the plane (not necessarily bicolored) and let $m$ be an integer such that $m\leq \frac{1}{2}|S|$. If there is a directed line $\ell$ through a point $x\in S$ such that $|left(\ell)\cap S|\leq m$, then there is a directed line, say $\ell'$, through $x$ such that $|left(\ell')\cap S|=m$. Moreover, $\ell'$ can be found in $O(N\log(N))$ time, where $N=|S|$.
\end{lemma}

\begin{proof}
Suppose that $|left(\ell)\cap S|=i\leq m\leq\frac{1}{2}|S|$. A continuous $\pi$-rotation of~$\ell$ through~$x$ starts with $|left(\ell)\cap S|=i\leq m$ and ends with $|left(\ell)\cap S|\geq |S|-i\geq\frac{1}{2}|S|\geq m$. Then, at some point during the rotation $|left(\ell)\cap S|=m$. To find such a rotation~$\ell'$ of the original line~$\ell$, we only need to order by slope the $N-1$ lines passing trough $x$ and each one of the remaining $N-1$ points of $S$ (this requires $O(N\log(N))$ time), and to explore these lines in order until finding $\ell'$ (this only requires linear time because when passing from a line to the next one, only a point passes from left to right of the explored line or vice versa).
\end{proof}

\section{Particular configurations}\label{sec-2}

In this section, we study some particular configurations of $(r,b)$-sets. As a first observation, notice that, if an $(r,b)$-set $S$ admits a $K_{1,3}$-covering consisting of $h\ge 0$ stars centered in red points and $g\ge 0$ stars centered in blue points, then $r=3g+h$, $b=3h+g$ and $|S|= 4(g+h)$. But this condition on the number of red and blue points is not sufficient to assure that a $K_{1,3}$-covering exists, as we show in Sections~\ref{subsec-2.1} and~\ref{subsec-2.5}. In Sections~\ref{subsec-2.4},~\ref{subsec-2.2} and~\ref{subsec-2.3}, we study some $(r,b)$-sets admitting a $K_{1,3}$-covering.

\subsection{Equitable sets}\label{subsec-2.4}

Let $S$ be an $(r,b)$-set such that either $r=3b$ or $b=3r$. In each of these cases we say that $S$ is a $\{1,3\}$-\emph{equitable set}. The following theorem, which is a generalization of the Ham-Sandwich theorem, allows us to show that any $\{1,3\}$-equitable set can be $K_{1,3}$-covered.

\begin{theorem} [Equitable Subdivision~\cite{BKS2000,BJ02,Ito2000,Sakai2002}]\label{th:C}
Let $c$, $d$ and $g$ be positive integers. If $cg$ red points and $dg$ blue points are given in the plane in general position, then there exists a subdivision of the plane into $g$ convex regions such that each region contains precisely $c$ red points and $d$ blue points. This subdivision can be computed in $O(N^{\frac{4}{3}}\log^2(N)\log(g))$ time, where $N=(c+d)g$.
\end{theorem}

Note that, in fact, the running time stated in~\cite{BKS2000} is $O(N^{\frac{4}{3}}\log^3(N)\log(g))$, because the authors used the dynamic algorithm by Overmars and van Leeuwen~\cite{OL1981} tho maintain a convex hull
in $O(\log^2(N))$ time per update. This was later improved by Brodal and Jacob~\cite{BJ02} to $O(\log(N))$ time per update. Therefore, the algorithm from Bespamyatnikh et al.~\cite{BKS2000} is directly improved by a logarithmic factor to $O(N^{\frac{4}{3}}\log^2(N)\log(g))$ time.

\begin{theorem}\label{th:3.4}
If an $(r,b)$-set $S$ is $\{1,3\}$-equitable, then all the points of $S$ can be $K_{1,3}$-covered, and a $K_{1,3}$-covering can be computed in $O(N^{\frac{4}{3}}\log^3(N))$ time where $N=r+b$.
\end{theorem}

\begin{proof}
Without loss of generality, we can assume that $S$ consists of $3b$ red points and $b$ blue points. We apply directly Theorem~\ref{th:C} by taking $c=3$, $d=1$ and $g=b$, so there exists a subdivision of the plane into $b$ convex regions such that each region contains precisely $3$ red points and $1$ blue point. The points in each region can be trivially $K_{1,3}$-covered and the union of these coverings is a $K_{1,3}$-covering of $S$ without crossings. As finding an equitable subdivision requires $O(N^{\frac{4}{3}}\log^2(N)\log(b))$ time, then building a $K_{1,3}$-covering takes $O(N^{\frac{4}{3}}\log^3(N))$ time.
\end{proof}

\subsection{Linearly separable sets}\label{subsec-2.2}


An $(r,b)$-set $S=R\cup B$ is \emph{linearly separable} if there exists a line $\ell$ that separates $ R $ and $ B $, say $R\subset left(\ell)$ and $B\subset right(\ell)$. For $r=3g+h$ and $b=3h+g$, any linearly separable set admits a $K_{1,3}$-covering, as the following theorem shows.

\begin{theorem}\label{th:3.2}
If $S=R\cup B$ is a linearly separable $(3g+h,3h+g)$-set, then all the points of $S$ can be $K_{1,3}$-covered.
\end{theorem}

\begin{proof}
Without loss of generality, assume $ g>0 $. The proof is by induction on $h$. If $h=0$, then $S$ is a $\{1,3\}$-equitable set, which can be $K_{1,3}$-covered by Theorem~\ref{th:3.4}. Assume $h>0$, and suppose that $R$ and $B$ are separable by the line $\ell$. Then, there are lines $\ell_b$ and $ \ell_r $ parallel to $ \ell $ such that $ \ell_b$ separates $4$ blue points from the rest of $S$ and $\ell_r$ separates $4$ red points from the rest of~$S$. See Figure~\ref{fig:3.22}.

\begin{figure}[ht]
    \centering
    \includegraphics[scale=1]{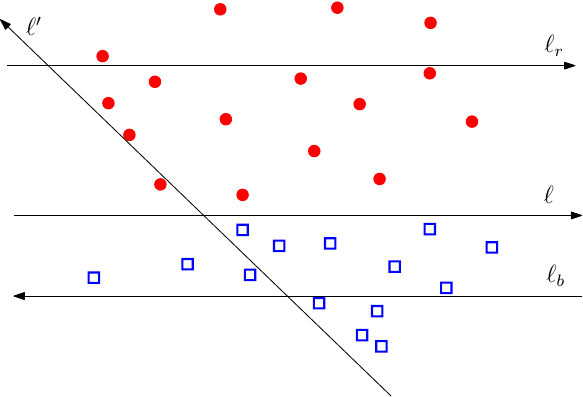}
    \caption{Illustration for the proof of Theorem~\ref{th:3.2}.}\label{fig:3.22}
\end{figure}

For these $\ell_b$ and $ \ell_r $, by Lemma~\ref{lem:9}, there is a line $\ell'$ that separates $4$ points of $S$, exactly~$3$ of them blue, from the rest of $S$. This set of $4$ points can clearly be $K_{1,3}$-covered. The rest of $S$ is a $(3g+(h-1),3(h-1)+g)$-set, still linearly separable, and thus, since $\ell'$ separates the relevant four points from~$S$, it can also be $K_{1,3}$-covered by induction. The two covers are separated by $\ell'$ and, thus, they together form a $K_{1,3} $-cover of $S$.
\end{proof}

Next, we show how to compute a $K_{1,3}$-covering for a linearly separable $(3g+h,3h+g)$-set in $O(N\log(N))$ time, where $N=4(g+h)$. The algorithm first computes $h$ stars centered in red points and then $g$ stars centered in blue points. The algorithm computes~$conv(S)$ and takes the leftmost edge which joins a red point $p_1$ with a blue point $q_1$ (see Figure~\ref{fig:3.2}). Then, it removes $q_1$, computes $conv(S\setminus q_1)$, takes the new leftmost edge $p_2q_2$ connecting the red point $p_2$ to the blue point $q_2$ ($p_2$ can coincide with $p_1$), removes $q_2$ and computes $conv(S\setminus \{q_1,q_2\})$. The new leftmost edge is then $p_3q_3$, being $p_3$ red and $q_3$ blue ($p_3$ can coincide with $p_2$). Observe that the line $\ell'$ supporting the edge $p_3q_3$ separates $q_1$ and $q_2$ from the rest of the points, so the star centered in $p_3$, connecting $p_3$ to $q_1, q_2$ and $q_3$, will not cross any other star formed with the rest of the points. This step finishes by defining the star formed by $p_3$, $q_1$, $q_2$ and $q_3$ and removing $p_3$ and $q_3$. Redefining $R:=R\setminus\{p_3\}$, $B:=B\setminus\{q_1,q_2,q_3\}$, the algorithm proceeds recursively building $h$ stars centered in red points. The $g$ stars centered in blue points are computed in a similar manner.

\begin{figure}[ht]
    \centering
    \includegraphics[scale=1]{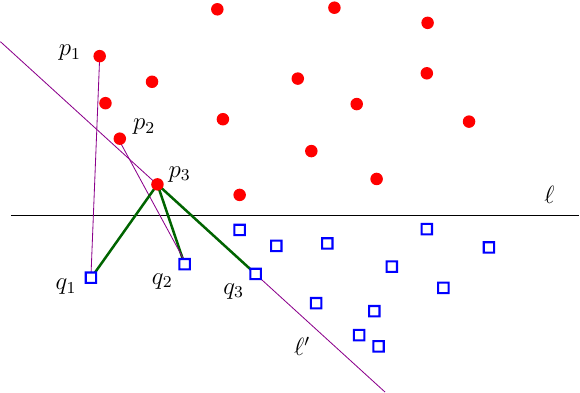}
    \caption{Illustration of the algorithm to build a $K_{1,3}$-covering for a linearly separable set.}\label{fig:3.2}
\end{figure}

The algorithm runs in $O(N\log(N))$ time. Computing $conv(S)$ requires $O(N\log(N))$ time. Besides, in the overall process, the algorithm removes the $N$ points, updating the convex hull after each removal. This process of updating $N$ times the convex hull requires $O(N\log(N))$ time, using the semi-dynamic data structure by Hershberger and Suri~\cite{HS1992}, which allows to process a sequence of $N$ deletions in $O(\log(N))$ amortized time per deletion. Let us see that, when building the stars centered in red points, all leftmost edges can be computed in linear time. Suppose that $p_iq_i$ is the leftmost edge in a generic step and $p_{i+1}q_{i+1}$ is the leftmost edge after removing $q_i$. Computing $p_{i+1}$ can be done by exploring in order the points from $p_i$ to $p_{i+1}$ on the boundary of the convex hull of the current set at this step of the algorithm, which are necessarily red. Moreover, if $p_{i+1}q_{i+1}$ is the leftmost edge after removing $p_i$ and~$q_i$, then $p_{i+1}$ can be computed by exploring again in order the points on the boundary of the convex hull of the current set from the  point preceding~$p_i$ to $p_{i+1}$. Then, computing all the leftmost edges only depends on the number of times that red points are explored by the algorithm: A red point is explored once if it appears on the boundary of the convex hull, another time if it is removed, and several times in the different steps if it belongs to the current leftmost edge or it is the red point previous to a removed red point. As the number of  leftmost edges required to define the $h$ stars centered in red points is the number of removed points ($4h$), then the number of times that red points are explored is $O(N)$. Using a similar reasoning, all the leftmost edges to build the $g$ stars centered in blue points can be computed in $O(N)$ time. Therefore, we have the following theorem.

\begin{theorem}\label{th:3.2b}
If $S=R\cup B$ is a linearly separable $(3g+h,3h+g)$-set, then a $K_{1,3}$-covering can be computed in $O(N\log(N))$ time, where $N=4(g+h)$.
\end{theorem}

\subsection{Convex point sets}\label{subsec-2.1}
In the previous subsections, we have shown some $(3g+h,3h+g)$-sets which always admit a $K_{1,3}$-covering. Now we show that this is not always the case for $(3g+h,3h+g)$-sets in convex position.

\begin{theorem}\label{th:3.1}
If $S=R\cup B$ is a $(3g+h,3h+g)$-set in convex position, then at least $4(g+h)-4$ points of $S$ can be $K_{1,3}$-covered and this bound is tight. 
\end{theorem}

\begin{proof}
The proof is again by induction on $|R\cup B|$. If $g=0$ or $h=0$, which includes the base case $|R\cup B|=4$, then $S$ is a $\{1,3\}$-equitable set, which can be $K_{1,3}$-covered by Theorem~\ref{th:3.4}. Thus, we assume $g,h>0$, $|R\cup B|\ge 8$ and the elements of $R\cup B$ are ordered clockwise along (the boundary of) $\operatorname{conv}(R\cup B)$.

\begin{itemize}
\item[(i)] Suppose that there exists a set $X\subset S$ of four consecutive points cyclically on $\operatorname{conv}(S)$ such that 3 of them have the same color and the remaining one has a distinct color. Then, the four points of $X$ can be $K_{1,3}$-covered, and $4(g+h-1)-4$ points of $(R\cup B)\setminus X$, which is either a $(3g+(h-1),3(h-1)+g)$-set or a $(3(g-1)+h,3h+(g-1))$-set, can be $K_{1,3}$-covered because of the induction hypothesis. Note that, by convexity, these two $K_{1,3}$-coverings do not cross, so the desired covering is obtained.

\item[(ii)] Suppose that such a set $X$ does not exist. By a simple counting argument, this implies one of these two cases: One red point and one blue point alternately lie on $\operatorname{conv}(S)$, or two red points and two blue points alternately lie on $\operatorname{conv}(S)$ (see Figure~\ref{fig:3.1}). Therefore, in both cases, $|R|=|B|$, $g=h$, $|R|=4g=|B|$, and $|S|=|R|+|B|=8g$. Assume that the $n=8g$ points of $S$ are numbered from $1$ to~$n$ clockwise around $\operatorname{conv}(S)$.

\begin{itemize}
\item[(1)] If one red point and one blue point alternate on $\operatorname{conv}(S)$, we obtain the desired $K_{1,3}$-covering as follows: For $k=1,\dots,2g-1$, we take the points with numbers $3k-2$, $3k-1$, $3k$, and $n-k$ and cover them with a $K_{1,3}$ (see Figure~\ref{fig:3.1}, left). Doing this, we leave uncovered precisely the four points numbered $6g-2$, $6g-1$, $6g$, and $n$.
\item[(2)] If two red points and two blue points alternate on $\operatorname{conv}(S)$, we can obtain the desired $K_{1,3}$-covering as follows: For $k=1,\dots,2g-1$, we take the points with numbers $3k-2$, $3k-1$, $3k$, and $n-k-1$ and cover them with a $K_{1,3}$ (see Figure~\ref{fig:3.1}, right). Doing this, we leave uncovered precisely the four points numbered $6g-2$, $6g-1$, $n-1$, and $n$.
\end{itemize}
\end{itemize}

\begin{figure}[htb]
\begin{center}
\includegraphics[scale=1]{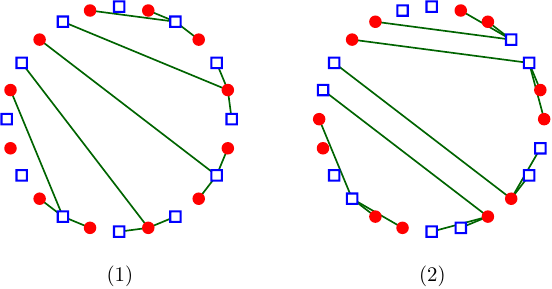}
\end{center}
\caption{Two configurations of points in convex position not admitting a $K_{1,3}$-covering.}\label{fig:3.1}
\end{figure}

The fact that the bound $4(g+h)-4$ is the best possible comes from the two previous particular configurations in~(1) and~(2).
In any full covering of a (monochromatic) convex point set, there is always a star that covers four consecutive points. However, in either of the configurations in~(1) and~(2), among any four consecutive points two are red and two are blue, hence they cannot be $K_{1,3}$-covered.
\end{proof}

We can cover at least $4(g+h)-4$ points in linear time as follows: Explore the points in order along (the boundary of) $conv(S)$, looking for four consecutive points, say $i, i+1, i+2$, and $i+3$, three of one color and one of the other color. If such a set exists, then define a star with these four points, remove them, and continue the process from point $i-3$, checking points $i-3$, $i-2$, $i-1$, and $i+4$. If such a set does not exist, then either one red point and one blue point alternately lie on $conv(S)$ or two red points and two blue points alternately lie on $conv(S)$, and the covering is built as shown in Figure~\ref{fig:3.1}. If, during the process, $h$ (resp. $g$) stars centered in red (resp. blue) points are defined, then the resulting configuration after removing all the stars defined is a $\{1,3\}$-equitable set formed by $3g'$ red points and $g'$ blue points ($3h'$ blue points and $h'$ red points). In this case, the process continues looking only for subsets $X$ consisting of four consecutive points, three of them red and the other blue (resp. three blue and one red if the equitable set is formed by $3h'$ blue points and $h'$ red points). Such subsets $X$ always exist using the following observation that easily follows from a simple counting argument.

\begin{observation}\label{obs:fourelements}
Let $(a_1, a_2, \ldots , a_n)$ be a cyclic sequence consisting of $b$ blue elements and $3b$ red elements. Then, there exist four consecutive elements in the sequence such that three of them are red and the other is blue.
\end{observation}

In addition to the two special configurations shown in Figure~\ref{fig:3.1} not admitting a $K_{1,3}$-covering, there are many more. For instance, take the configuration where one red point and one blue point alternately lie on $\operatorname{conv}(S)$ and add four points in the order red, blue, red, red between a blue point and a red point. The new configuration does not admit a $K_{1,3}$-covering, and this operation can be repeated several times.

\begin{figure}[ht]
\begin{center}
\includegraphics[scale=1]{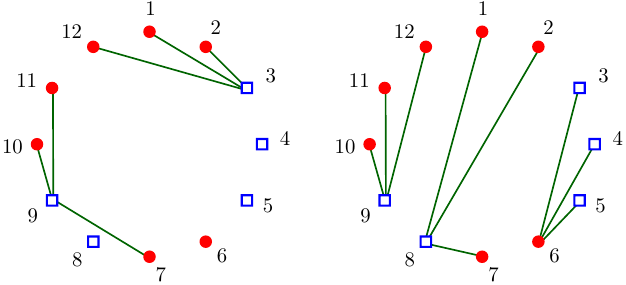}
\end{center}
\caption{If a star is defined using points $12, 1, 2$ and $3$ (left), then it is not possible to cover the rest of the points. But if the star is defined using points $9,10,11$ and $12$ (right), then the rest of the points can be covered.}\label{fig:3.1a}
\end{figure}

On the other hand, given a bicolored convex point set, depending on the order in which the stars are computed using the previous algorithm, sometimes a $K_{1,3}$-covering is found and sometimes not. See Figure~\ref{fig:3.1a}. This leads to the problem of deciding if a bicolored convex point set admits a $K_{1,3}$-covering or not. Using a standard dynamic programming algorithm, this decision problem can be solved in $O(n^3)$ time.
Assume that the points are numbered clockwise from 1 to $n$. For an interval $[i,j]$ of consecutive points (mod $n$), define $c(i, j, u_r, u_b)$, for non-negative integers $u_r$ and $u_b$ such that $u_r+u_b \leq 2$, to be $true$ if we can $K_{1,3}$-cover the points from $i$ to $j$ leaving $u_r$ red and $u_b$ blue points from $[i, j]$ uncovered and in such a way that no edge from the $K_{1,3}$-covering crosses any segment between one of those $u_r+u_b$ uncovered points and a point outside~$[i,j]$.

To calculate $c(i, j, u_r, u_b)$ we need to consider the following possibilities:
\begin{enumerate}
\item The point~$i$ is left uncovered. Then, $c(i, j, u_r, u_b) = c(i + 1, j, u_r-1, u_b)$, if the point $i$ is red and $u_r > 0$, or $c(i, j, u_r, u_b) = c(i+1, j, u_r, u_b-1)$, if the point $i$ is blue and $u_b > 0$.
\item The point $i$ is part of a $K_{1,3}$-star contained in~$[i, j]$. Let $k$ be the last vertex of that star when going from $i$ to $j$ in the direction of increasing indexes. We will connect $i$ with $k$, and ask for a solution of the subproblem from $i+1$ to $k-1$ which leaves two of the vertices free to be connected with points~$i$ and~$k$.
\begin{enumerate}
    \item If we take the point $i$ or the point $k$ to be the center of the star, then
\[
c(i, j, u_r, u_b) =
\!\!\!\!\!\!\!\!\!\!\!\!\!\!\!\!
\bigvee_{\small
\begin{array}{c}k=i+3\\\mbox{\small
$i,k$ different colors}\end{array}
}^{j-(u_r+u_b)}
\!\!\!\!\!\!\!\!\!\!\!\!\!\!\!\!
\big(
c(i+1, k-1, 2, 0)
\lor
c(i+1, k-1, 0, 2)
\big) \land c(k+1, j, u_r, u_b)
\]
\item If we take both points $i$ and $k$ to be  leaves of the star, then
\[
c(i, j, u_r, u_b) =
\bigvee_{\small
\begin{array}{c}k=i+3\\\mbox{\small $i,k$ same color}\end{array}
}^{j-(u_r+u_b)}
c(i+1, k-1, 1, 1)
\land c(k+1, j, u_r, u_b)
\]
\end{enumerate}
\end{enumerate}

There are $O(n^2)$ subproblems in total, and solving each of them takes $O(n)$ time, yielding an $O(n^3)$ algorithm. Therefore, we have the following theorem.


\begin{theorem}\label{th:algorithmconvex}
If $S=R\cup B$ is a $(3g+h,3h+g)$-set in convex position, deciding if $S$ admits a $K_{1,3}$-covering can be done in $O(n^3)$ time and $O(n^2)$ space. Moreover, at least $4(g+h)-4$ points of $S$ can be $K_{1,3}$-covered in linear time.
\end{theorem}

\subsection{Double chain}\label{subsec-2.3}

In this section we show that, by combining two bicolored convex point sets in a special way, it becomes possible to cover all points. A \emph{double chain}~\cite{GNT2000} is formed by two (non-empty) convex chains $C_1$ and $C_2$ such that every edge $(p,q)$ with $p\in C_1$ and $q\in C_2$ does not cross neither $conv(C_1)$ nor $conv(C_2)$. See Figure~\ref{fig:2chain-1}.

Points in $C_1$ are denoted by $p_1,\ldots,p_s$ counter-clockwise along $conv(C_1)$ and points in~$C_2$ are denoted by $q_1,\ldots,q_t$ counter-clockwise along $conv(C_2)$. Let us denote by $circ(C_1)$, $circ(C_2)$, and $circ(C_1\cup C_2)$ the circular sequences $p_1,\ldots,p_s$; $q_1,\ldots,q_t$; and $p_1,\ldots,p_s,q_1,\ldots,q_t$, respectively. For a bicolored double chain, there is always a $K_{1,3}$-covering, as the following theorem shows.

\begin{figure}[ht]
    \centering
    \includegraphics[scale=1]{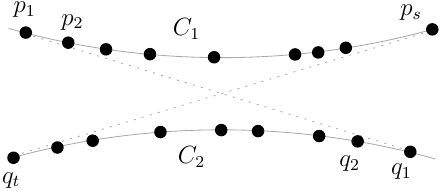}
    \caption{A double chain.}
    \label{fig:2chain-1}
\end{figure}

\begin{theorem}\label{th:3.3}
If $R\cup B$ is a double chain, with $|R|=3g+h$ and $|B|=g+3h$ for some integers $g,h\ge 0$, then $R\cup B$ can be $K_{1,3}$-covered.
\end{theorem}

\begin{proof}
We prove the theorem by induction on $|R\cup B|$, so that in any inductive step we guarantee that $R\cup B$ is a double chain. If $g=0$ or $h=0$, which includes the case $|R\cup B|=4$, then $R\cup B$ is a $\{1,3\}$-equitable set, which can be $K_{1,3}$-covered by Theorem~\ref{th:3.4}. Then, assume $g,h>0$, $|R\cup B|\ge 8$, and without loss of generality\ $0<|C_1|\le |C_2|$ and $p_1$ is red.

As $|R\cup B|\ge 8$, then $|C_2| \ge 4$. If $circ(C_2)$ has a set $X$ consisting of $4$ consecutive points, 3 of them having the same color and the remaining one having a distinct color, then $X$ can be $K_{1,3}$-covered and the set $(R\cup B)\setminus X$ is a double chain with four points less (except if $|C_2| = 4$) which can also be $K_{1,3}$-covered by the induction hypothesis. Note from the definition of double chain that any edge in the covering of $(R\cup B)\setminus X$ is disjoint from any edge in the covering of $X$, thus $R\cup B$ can be $K_{1,3}$-covered. If $|C_2| = 4$ and $circ(C_2)$ has such a set $X$, then $C_1$ has exactly four points, three having the same color and the fourth having the other color, and these points can be $K_{1,3}$-covered without crossing $conv(C_2)$.

\begin{figure}[ht]
    \centering
    \includegraphics[width=\textwidth]{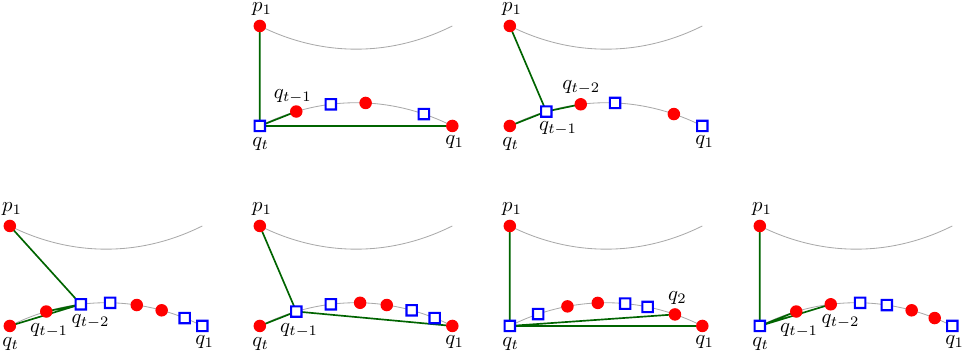}
    \caption{Constructing a $K_{1,3}$-covering for the first point of $C_1$ and $3$ consecutive
    points of $C_2$, including the last point of $C_2$.}
    \label{fig:2chain-3}
\end{figure}

Assume, to the contrary, that $circ(C_2)$ does not have any such set $X$. Then, either the colors of the points in $circ(C_2)$ alternate as described in the proof of Theorem~\ref{th:3.1} (see also Figure~\ref{fig:3.1}) or $C_2$ consists of points of the same color. If the colors alternate in $circ(C_2)$, there are six possible ways of coloring $\{q_t,q_{t-1},q_{t-2}\}$ in red and blue, namely: $(red,red,blue)$, $(red,blue,red)$, $(red,blue,blue)$, $(blue,blue,red)$, $(blue,red,blue)$ and $(blue,red,red)$, see Figure~\ref{fig:2chain-3}. Moreover, $|C_2|$ is even, implying that $|C_1|\ge 2$ since $|R\cup B|=|C_1|+|C_2|$ is also even. In any case, it is always possible to ensure that one point set, denoted by $Z$, among $\{p_1,q_t,q_{t-1},q_{t-2}\}$, $\{p_1,q_t,q_{t-1},q_1\}$, and $\{p_1,q_t,q_1,q_2\}$, can be $K_{1,3}$-covered and satisfies that any segment connecting points in $(R\cup B)\setminus Z$ and any edge of the covering are disjoint. Figure~\ref{fig:2chain-3} shows the $K_{1,3}$-coverings for the different sets~$Z$. Then, we can apply induction on $(R\cup B)\setminus Z$ to obtain a $K_{1,3}$-covering for $R\cup B$.

Finally, assume that $C_2$ consists of points of the same color. As $g,h >0$, then $|R|\ge 4$ and $|B|\ge 4$, so $|C_1|\ge 4$ because $C_2$ is monochromatic. Suppose first that $C_2$ is red, so there are at least four blue points in $C_1$. If $p_i$ is the first blue point in $circ(C_1\cup C_2)$, starting the exploration at $p_1$, then the three points appearing in $circ(C_1\cup C_2)$ before $p_i$ are red, so the set $X$ consisting of these four points can be $K_{1,3}$-covered. Note that $(R\cup B)\setminus X$ is a double chain and that the $K_{1,3}$-covering of $X$ does not cross any $K_{1,3}$-covering of $(R\cup B)\setminus X$ which exists by induction. Suppose now that $C_2$ is blue. As $p_1$ is red, then $X=\{p_1,q_t,q_{t-1},q_{t-2}\}$ can be $K_{1,3}$-covered without crossing a $K_{1,3}$-covering of the double chain $(R\cup B)\setminus X$ that exists by induction.
\end{proof}

From an algorithmic point of view, one can design a linear time algorithm to find a $K_{1,3}$-covering for a double chain as follows. We omit the details and give only a general idea on how the algorithm works. Let $R'$ and $B'$ be the sets of red and blue points forming a double chain in a generic step of the algorithm, with $|R|=3g'+h'$, $|B|=g'+3h'$, $g'\le g$ and $h'\le h$. The algorithm controls if either $g'=0$ or $h'=0$. If this is the case, then the double chain is a $\{1,3\}$-equitable set. By Observation \ref{obs:fourelements}, $circ(C_1\cup C_2)$ necessarily contains a set $X$ consisting of 4 consecutive points, three of them red (resp. blue) and the other one blue (resp. red). Therefore, the linear algorithm described in the convex case can be easily adapted to the circular sequence $circ(C_1\cup C_2)$ to find a $K_{1,3}$-covering for $R'\cup B'$.

From the beginning, the algorithm runs twice the linear time algorithm described for the convex case, once to control a possible set $X$ of 4 consecutive points, three of one color and one of the other color, in $C'_1$ and a second time to control such a set $X$ in $C'_2$. One of these two algorithms is activated in each step, depending on the sizes of $C'_1$ and $C'_2$. If such a set $X$ does not exist, then the algorithm has to apply two different procedures: One to define the different stars when $C'_2$ is monochromatic (as described in the proof of Theorem \ref{th:3.3}) and the other one to define the stars when the points in $C'_2$ alternate (looking for stars involving $p'_1$ and $q'_t$ as described in Figure~\ref{fig:2chain-3}).

\subsection{Some $(r,b)$-sets with $\mathcal{U}(r,b)\ge5$}\label{subsec-2.5}

In this section we give some particular configurations of $(r,b)$-sets for which $\mathcal{U}(r,b)\ge 5$ and prove that, except for $\{1,3\}$-equitable sets, $\mathcal{U}(r,b)>0$ for any values of $r$ and $b$. Recall that $\mathcal{U}(r,b)$ is the maximum of $\mathcal{U}(S)$ over all $(r,b)$-sets $S$ in general position, being $\mathcal{U}(S)$ the minimum number of points in $S$ that cannot be $K_{1,3}$-covered.

\begin{theorem}\label{lem:PerfectCoverings}
Let $ r $ and $ b $ be nonnegative integers.
\begin{itemize}
\item [(a)] $\mathcal{U}(r,b)= 0$ if, and only if, $r=3b$ or $b=3r$.


\item [(b)] $\mathcal{U}(2k+1,2k)\geq 5$ for any $ k\geq 4 $.
\end{itemize}
\end{theorem}

\begin{figure}[ht]
\centering
\includegraphics[scale=1]{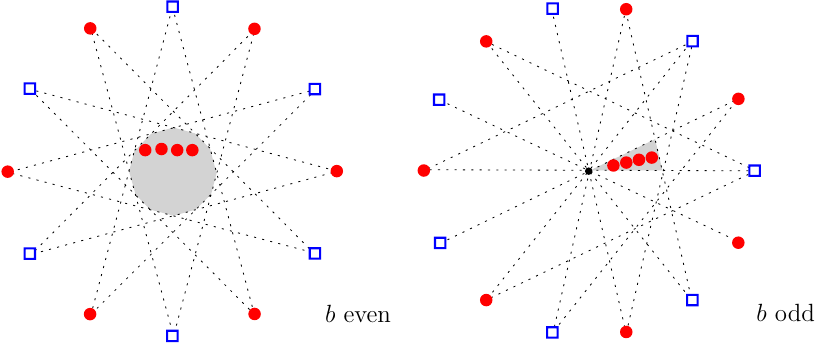}
\caption{Cases $ b $ even and $ b $ odd in the proof of Theorem~\ref{lem:PerfectCoverings}(a).}
\label{fig:Fig4N_even_odd}
\end{figure}

\begin{proof}
\noindent {(a)} Assume that $ r\geq b $ and $r\neq 3b$. If $ r > 3b $, a perfect covering of any $ (r,b) $-set  is numerically impossible as there are too many red points (even using only stars with blue center would leave red points uncovered). If $ r<3b $, the following $(r,b)$-set $ S $ cannot completely be $K_{1,3}$-covered: Draw a color-alternating convex $2b$-gon together with a set~$ Q $ of $r-b$ almost collinear red points as illustrated in Figure~\ref{fig:Fig4N_even_odd}. The set $ Q $ is not split by any bichromatic diagonal and its points are almost along a line $ \ell $ through the midpoints of two antipodal sides. Note that in a perfect $ K_{1,3} $-covering there must be at least one star~$ T $ with red center, since $ r<3b $. If this center of $ T $ is in $ Q $, then $ T $ must have at least two consecutive edges on the same side of $\ell$ (see Figure~\ref{fig:Fig4N_proof}, left). But these two edges isolate an odd number of points in $ S $ and, therefore, these points cannot be $ K_{1,3} $-covered. If the red center of $T$ is not in $Q$, then the set~$Q$ must be either to the left or
to the right of at least two edges of~$T$ (see Figure~\ref{fig:Fig4N_proof}, right). As before, two consecutive such edges isolate an odd number of points in $ S $ from the rest and so they cannot be completely $ K_{1,3} $-covered. If $ r=3b $, then Theorem~\ref{th:3.4} guarantees $\mathcal{U}(r,b)= 0$.

\begin{figure}[ht]
\centering
\includegraphics[scale=1]{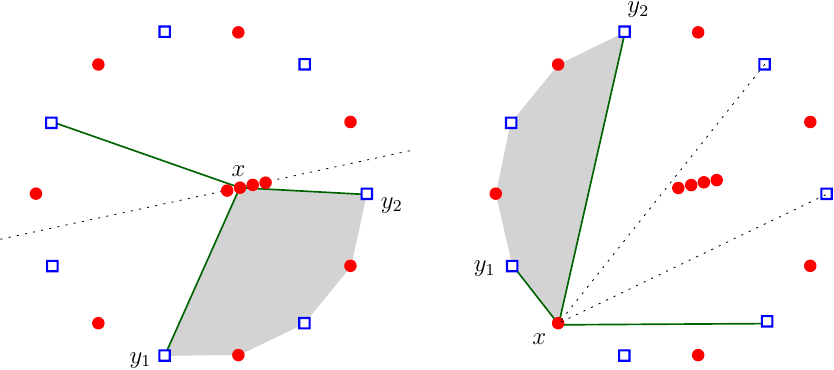}
\caption{Red center in $Q$ (left) or on the convex hull (right).}
\label{fig:Fig4N_proof}
\end{figure}



\begin{figure}[htb]
\centering
\includegraphics[width=\textwidth]{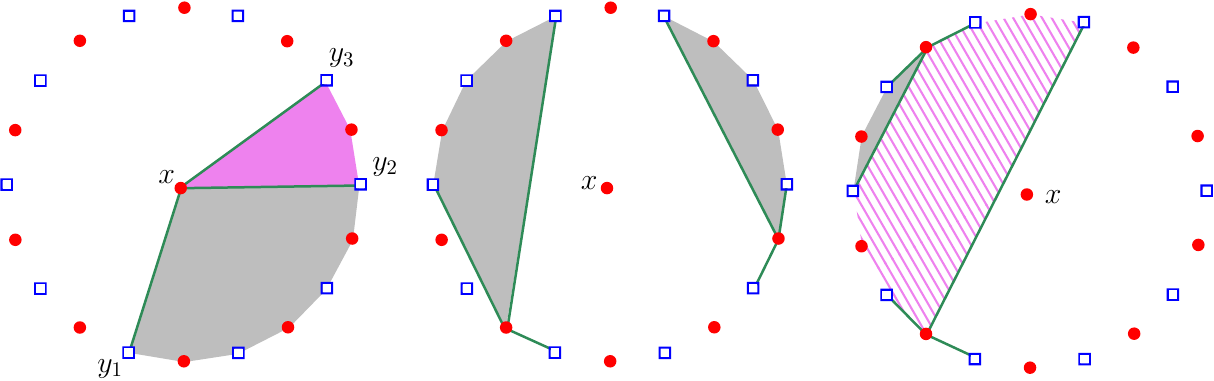}
\caption{Red center in $Q$ or two on the convex hull.}
\label{fig:Fig5left_proof}
\end{figure}

\noindent {(b)} Consider the $ (2k+1,2k) $-set $ S $ formed by the color-alternating vertices of a regular $ 4k $-gon and one red point $ x $ near the center of the polygon. Since $\mathcal{U}(2k+1,2k)\equiv 1 \pmod {4}  $, it is enough to show that $ \mathcal{U}(S)>1 $. Note that in a $ K_{1,3} $-covering of $ S $ leaving only one point uncovered, there must be exactly $ \lfloor k/2 \rfloor\geq 2$ stars with red center. If there is a star $ T $ whose center is $ x $, each of the sectors determined by the edges of $ T $ contains an odd number of points (see Figure~\ref{fig:Fig5left_proof}, left).
Thus, each of these subsets of $ S $ cannot be completely $ K_{1,3} $-covered.

If there are two stars with red centers that are not $ x $, then two cases appear depending on whether $x$ is between the two stars (Figure~\ref{fig:Fig5left_proof}, center) or it is in one of the regions delimited by two consecutive edges (diagonals) of one of the stars (Figure~\ref{fig:Fig5left_proof}, right). In the first case, each star defines a region isolating an odd number of points. In the second case, one of the stars defines a region with an odd number of points and the region delimited by the two stars contains $2j+2$ points, for some $j$, which is not a multiple of 4.  In either case there are at least two points that are not covered.
\end{proof}

\section{General configurations of points}\label{sec-3}

Given an $(r,b)$-set $S$, the main goal of this section is to give a lower bound on the number of points of $S$ that can always be $K_{1,3}$-covered. Assume that $b\le r$ and let $\alpha=\frac{b}{r}$. Notice that, if $\alpha < 1/3$, then $r> 3b$ and there are too many red points to be covered, even by stars centered in blue points. In this case, at least $r-3b$ red points will necessarily remain uncovered in any $K_{1,3}$-covering. Moreover, the bound is tight because, by removing $r-3b$ red points, we obtain a $\{1,3\}$-equitable set that can be $K_{1,3}$-covered by Theorem~\ref{th:3.4}.

When $1/3 \le \alpha \leq 1$, our first attempt to obtain a ``good covering" was trying to divide the plane into disjoint convex regions such that every region contained either three red points and one blue point or vice versa. If such a partition exists, then the four points in each region can be trivially $K_{1,3}$-covered. Unfortunately, this is not always possible. For instance, a color-alternating convex point set does not admit this kind of partition. Using a different (but in same way similar) approach, we prove that, when $1/3 \le \alpha \leq 1$, at least $\frac{8}{9}(r+b-8)$ points can always be $K_{1,3}$-covered. The proof of this result is divided into two parts: Section~\ref{sect:lowerrange} proves the result for $1/3 \leq \alpha \leq  4/5$ and Section~\ref{sect:upperrange} for $4/5 \leq \alpha \leq 1$. Finally, in Section~\ref{sect:complexity}, we show an efficient algorithm for computing at least $\frac{8}{9}(r+b-8)$ points that can be always $K_{1,3}$-covered.

\subsection{Lower bound when $1/3 \le \alpha \le 4/5$}\label{sect:lowerrange}

{In this case, the general idea to prove that at least $\frac{8}{9}(r+b-8)$ points can always be $K_{1,3}$-covered is the following. We first exhibit a special family of $(r,b)$-sets such that for any member of this family it is possible to bound the number of uncovered points (Theorem~\ref{th:10} and Lemma~\ref{lem:11}). Then, we show that any $(r,b)$-set can be transformed into a member of this family by removing some red and blue points, so that a bound on the number of uncovered points can be given for any $(r,b)$-set (Theorem~\ref{th:B-S}). Using this result, we show that at least $\frac{8}{9}(r+b-8)$ points can always be $K_{1,3}$-covered when $1/3 \le \alpha \le 4/5$ (Corollary~\ref{cor:8_9_Lower}). Let us go into the details.}

{Given a bicolored point set $S=R\cup B$, the following theorem gives an upper bound on the number of points that will remain uncovered, under certain constrains on the number of red and blue points. This theorem is the main result of this section and its proof, included in Section~\ref{sect:proof}, requires the two lemmas in Section~\ref{sect:lemmas} together with a new one (Lemma~\ref{lem:11}).}

\begin{theorem}\label{th:10}
Let $t$ be a  integer, $t\ge 0$. Then, $\mathcal{U}(3k-t,k+2t)\leq t$ for any integer $k\geq\frac{5}{8}t$.
\end{theorem}

Theorem~\ref{th:10} implies the following lower bound on $  \mathcal{C}(r,b)$ for all $ \alpha $ with $\frac{1}{3}\leq\alpha\leq 1$. We use this bound to prove the general lower bound $  \mathcal{C}(r,b)\geq \frac{8}{9}(r+b)-4$ for any $\frac{1}{3}\leq\alpha\leq \frac{4}{5}$ as stated in Corollary \ref{cor:8_9_Lower}.

\begin{theorem}\label{th:B-S}
Let $r$ and $b$ be positive integers, $r \geq b$,
and let $\alpha=\frac{b}{r}$.
If $\frac{1}{3}\leq\alpha\leq 1$, then

\[
\mathcal{C}(r,b)\geq\frac{4}{7}\left(\frac{\alpha+2}{\alpha+1}\right)(r+b)-4.
\]

\end{theorem}

\begin{proof}
Write $3b-r=7t+s$ for some integers $t\geq 0$ and $s$ with $0\leq s\leq 6$. Let $k=b-2t-\lceil\frac{s}{3}\rceil$. Then, $r=3k-t+3\lceil\frac{s}{3}\rceil-s$ and $b=k+2t+\lceil\frac{s}{3}\rceil$. Also, $k\geq\frac{5}{8}t$: For $t=0$, $3b-r=s\ge 0$, implying  $k=b-\lceil s/3 \rceil \ge 0$, so $k\ge 5/8 t$. For $t > 0$, as $b \le r$, then $3b-r \le 2b$, implying  $7t+s \le 2b$. Hence, $3t+s-2\lceil s/3 \rceil \le 2b - 4t - 2\lceil s/3 \rceil = 2 k$ and $k\ge (3t+s-2\lceil s/3 \rceil)/2$. For $0\le s\le 6$, the minimum of $s-2\lceil s/3 \rceil$ is reached at $s=1$, so $k\ge (3t-1)/2$. Since $t > 0$, $(3t-1)/2 \ge 5/8 t$.
Removing $3 \lceil s/3 \rceil -s$ red and $\lceil s/3 \rceil$ blue points,

\[
\mathcal{U}(r,b)=\mathcal{U}(3k-t+3\left\lceil\frac{s}{3}\right\rceil-s,k+2t+\left\lceil\frac{s}{3}\right\rceil) \leq \mathcal{U}(3k-t,k+2t)+4\left\lceil\frac{s}{3}\right\rceil -s
\]
\noindent and, by Theorem \ref{th:10},  $\mathcal{U}(3k-t,k+2t)+4\left\lceil\frac{s}{3}\right\rceil -s \leq t+4\left\lceil\frac{s}{3}\right\rceil -s$.
Then,
\begin{multline*}
\mathcal{U}(r,b)=\mathcal{U}\left(3k-t+3\left\lceil\frac{s}{3}\right\rceil-s,k+2t+\left\lceil\frac{s}{3}\right\rceil\right)\\
\leq \mathcal{U}(3k-t,k+2t)+4\left\lceil\frac{s}{3}\right\rceil -s\leq t+4\left\lceil\frac{s}{3}\right\rceil -s,
\end{multline*}
and, therefore,
\begin{multline*}
\mathcal{C}(r,b)=r+b-\mathcal{U}(r,b)\geq r+b-t-4\left\lceil\frac{s}{3}\right\rceil +s=4k\\
=4\left(b-\frac{2}{7}(3b-r-s)-\left\lceil\frac{s}{3}\right\rceil\right)=\frac{4}{7}(b+2r)+\frac{8}{7}s-4\left\lceil\frac{s}{3}\right\rceil
\geq\frac{4}{7}\left(\frac{\alpha+2}{\alpha+1}\right)(r+b)-4.
\end{multline*}
\end{proof}

\begin{corollary}\label{cor:8_9_Lower}
Let $r\ge b$ be positive integers such that $\frac{1}{3}\leq \alpha \leq\frac{4}{5}$. If $r$ red points and $b$ blue points are given in the plane in general position, then at least $\frac{8}{9}(r+b)-4$ points can be $K_{1,3}$-covered.
\end{corollary}

\begin{proof}
Since the function $\frac{\alpha+2}{\alpha+1}$ is decreasing, then for $\frac{1}{3}\leq\alpha\leq\frac{4}{5}$ we have
\[
\mathcal{C}(r,b)\geq\frac{4}{7}\left(\frac{\alpha+2}{\alpha+1}\right)(r+b)-4\geq\frac{4}{7}\left(\frac{\frac{4}{5}+2}{\frac{4}{5}+1}\right)(r+b)-4=\frac{8}{9}(r+b)-4.
\]
\end{proof}

The following lemma extends the range of validity of Theorem~\ref{th:10} and will be used as an extension of the basis of induction in the proof of Theorem~\ref{th:10}.

\begin{lemma}\label{lem:11}
If $k=\lceil\frac{5}{8}t\rceil-1$, then $\mathcal{U}(3k-t,k+2t)\leq\min\{8,t\}\leq t$ for any $t\geq 4$ or $t=2$.
\end{lemma}

\begin{proof}
Let $r=3k-t$ and $b=k+2t$. Since $\frac{5}{8}t-1\leq k=\lceil\frac{5}{8}t\rceil -1 <\frac{5}{8}t$, then
\[
\frac{21}{8}t-1\leq b <\frac{21}{8}t\leq 3r+9<\frac{21}{8}t+9.
\]
Then, $b=3r+i$ for some $0\leq i\leq 8$ and thus $\mathcal{U}(r,b)=\mathcal{U}(r,3r+i)\leq \mathcal{U}(r,3r)+i=i\leq 8$. Hence, the result holds for $t\geq 8$. If $t$=2, then $k=\lceil\frac{5}{8}t\rceil-1=1$, $r=1$ and $b=5$, implying that $i=2$ because $b=3r+i$. Moreover, if $t=4$, $5$, $6$, or $7$, then $i=4$, $1$, $6$, or $3$. Therefore $\mathcal{U}(r,b)\leq i\leq t$ in all cases.
\end{proof}

\subsubsection{Proof of Theorem~\ref{th:10}}\label{sect:proof}

We prove the result by induction on $t$. Theorem~\ref{th:3.4} shows the result for $t=0$. Also, $\mathcal{U}(2,3)=1$ shows the result for $t=k=1$. Assume that $t\geq 1$, $(t,k)\neq (1,1)$, and consider the set $R\cup B$ of points in general position such that $|R|=3k-t$ 
and $|B|=k+2t$
for some $k\geq\frac{5}{8}t$. The following cases arise, {depending on the parity of $t$ and $k$. The general idea in each case is to divide the original point set into two disjoint points sets, roughly involving half of the points, and apply induction on them.}

\begin{enumerate}
\item Suppose that $t$ and $k$ are even.

Let $k=2j$ for some integer $j$. By Ham-Sandwich theorem
and by induction (since $k=2j\geq\frac{5}{8}t$ {then} $j\geq\frac{5}{8}\cdot\frac{t}{2}$), we have
\[
\mathcal{U}(3k-t,k+2t)=\mathcal{U}(6j-t,2j+2t)\leq 2\ \mathcal{U}\left(3j-\frac{t}{2},j+2\cdot\frac{t}{2}\right)\leq 2\cdot\frac{t}{2}=t.
\]

\item Suppose that $t$ is even and $k$ is odd.

Let $k=2j+1$ for some integer $j$. Consider two parallel lines $\ell_1 $ and $\ell_2 $, with opposite directions, such that $|left(\ell_1)\cap (R\cup B)|=|left(\ell_2)\cap (R\cup B)|=4j+\frac{t}{2}$. By Lemma~\ref{lem:10}, we can further assume that there is at least one red point in $right(\ell_1)\cap right(\ell_2)$. (If $right(\ell_1)\cap right(\ell_2)$ only contains blue points, Lemma~\ref{lem:10} is applied taking $x$ as a red point in $left(\ell_1)$ and $m=4j+\frac{t}{2}+1$, and the line $\ell$ obtained according to the lemma is moved slightly such that it does not pass through $x$.) Since there are at most $|R|-1=6j+2-t$ red points in $left(\ell_1)\cup left(\ell_2)$, then we can assume without loss of generality that $|left(\ell_1)\cap R|\leq 3j+1-\frac{t}{2}$ and $|left(\ell_2)\cap R|\geq 3j+1-\frac{t}{2}$.

\begin{itemize}
\item[(a)] If $|left(\ell_1)\cap R|\leq 3j-\frac{t}{2}$, and since $|left(\ell_2)\cap R|\geq 3j+1-\frac{t}{2}$, by Lemma~\ref{lem:9} there exists a directed line $\ell$ such that $|left(\ell)\cap (R\cup B)|=4j+\frac{t}{2}$ and $|left(\ell)\cap R|=3j-\frac{t}{2}$. Then, $|left(\ell)\cap B|=j+2\cdot\frac{t}{2}$, where $\frac{t}{2}$ is a positive integer and $2j+1\geq\frac{5}{8}t$ implies that either $j+1>j\geq\frac{5}{8}\cdot\frac{t}{2}$ or $j+1>j\geq\lceil\frac{5}{8}\cdot\frac{t}{2}\rceil-1$ and $\frac{t}{2}\not=1,3$. By induction and Lemma~\ref{lem:11},
\[
{\mathcal{U}(left(\ell)\cap (R\cup B))}\leq\mathcal{U}\left(3j-\frac{t}{2},j+2\cdot\frac{t}{2}\right)\leq\frac{t}{2}.
\]

(Note that Lemma~\ref{lem:11} is necessary, for some values of $t$ and $k$, to substitute Theorem~\ref{th:10} in the inductive step. For example, if $t=4$ and $k=3$, then $j =1$ is less than $5/8 (t/2) = 10/8$, so Theorem~\ref{th:10} cannot be applied and we need Lemma~\ref{lem:11}, which can be applied because $j \ge \lceil (5/8)\cdot (t/2) \rceil -1$ holds.)

Also $|right(\ell)\cap R|=3(j+1)-\frac{t}{2}$ and $|right(\ell)\cap B|=(j+1)+2\cdot\frac{t}{2}$ and, by induction,
\[
{\mathcal{U}(right(\ell)\cap (R\cup B))}\leq \mathcal{U}\left(3(j+1)-\frac{t}{2},(j+1)+2\cdot\frac{t}{2}\right)\leq\frac{t}{2}.
\]

Therefore $\mathcal{U}(R\cup B)\leq 2\cdot\frac{t}{2}=t$.

\item[(b)] If $|left(\ell_1)\cap R|=3j+1-\frac{t}{2}=3j-(\frac{t}{2}-1)$, then $|left(\ell_2)\cap R|=3j-(\frac{t}{2}-1)$, $|left(\ell_1)\cap B|=|left(\ell_2)\cap B|=j+2(\frac{t}{2}-1)+1$, and $(right(\ell_1)\cap right(\ell_2))\cap (R\cup B)$ contains a red point and three blue points.
Moreover, $2j+1\geq\frac{5}{8}t$ implies $j\geq\frac{5}{8}(\frac{t}{2}-1)$. Therefore, by induction,
\begin{multline*}
\mathcal{U}(R\cup B)\leq 2\ \mathcal{U}\left(3j-\left(\frac{t}{2}-1\right),j+2\left(\frac{t}{2}-1\right)+1\right)+\mathcal{U}(1,3)\\
\leq 2\left(\left(\frac{t}{2}-1\right)+1\right)+0=t.
\end{multline*}
\end{itemize}

\item Suppose that $t$ is odd and $k$ is even.

Let $k=2j$ for some integer $j$. Consider two parallel lines $\ell_1$ and $\ell_2$, with opposite directions, such that $|left(\ell_1)\cap (R\cup B)|=|left(\ell_2)\cap (R\cup B)|=4j+\frac{t-1}{2}$. By Lemma~\ref{lem:10}, we can further assume that the unique point in $right(\ell_1)\cap right(\ell_2)$ is blue. Since all $6j-t$ red points are in $left(\ell_1)\cup left(\ell_2)$, then without loss of generality $|left(\ell_1)\cap R|\leq 3j-\frac{t+1}{2}$ and $|left(\ell_2)\cap R|\geq 3j-\frac{t-1}{2}$. Thus, Lemma~\ref{lem:9} guarantees the existence of a line~$\ell$ such that $|left(\ell)\cap (R\cup B)|=4j+\frac{t-1}{2}$ and $|left(\ell)\cap R|=3j-\frac{t-1}{2}$. Then, $|left(\ell)\cap B|=j+2\cdot\frac{t-1}{2}$, where $\frac{t-1}{2}$ is a nonnegative integer and $2j\geq\frac{5}{8}t$ implies that $j\geq \frac{5}{8}\cdot\frac{t-1}{2}$. By induction,
\[
{\mathcal{U}(left(\ell)\cap (R\cup B))}\leq \mathcal{U}\left(3j-\frac{t-1}{2},j+2\cdot\frac{t-1}{2}\right)\leq \frac{t-1}{2}.
\]
Also $|right(\ell)\cap R|=3j-\frac{t+1}{2}$ and $|right(\ell)\cap B|=j+2\cdot\frac{t+1}{2}$, where $\frac{t+1}{2}$ is a positive integer and $2j\geq\frac{5}{8}t$ implies that either $j\geq\frac{5}{8}\cdot\frac{t+1}{2}$ or $\frac{t+1}{2}\not=1, 3$ and $j\geq\lceil\frac{5}{8}\cdot\frac{t+1}{2}\rceil-1$. By induction and Lemma~\ref{lem:11},
\[
{\mathcal{U}(right(\ell)\cap (R\cup B))}\leq \mathcal{U}\left(3j-\frac{t+1}{2},j+2\cdot\frac{t+1}{2}\right)\leq\frac{t+1}{2}.
\]

Therefore $\mathcal{U}(R\cup B)\leq\frac{t-1}{2}+\frac{t+1}{2}=t$.

\item Suppose that $t$ and $k$ are odd.

Let $k=2j+1$ for some integer $j$ {(recall that $(k,j)\neq (1,1)$, so $j>0$)}. Consider two parallel lines $\ell_1$ and $\ell_2$, with opposite directions, such that
\[
|left(\ell_1)\cap (R\cup B)|=|left(\ell_2)\cap (R\cup B)|=4j+\frac{t-1}{2}.
\]

\begin{itemize}
\item[(a)] If one of $|left(\ell_1)\cap R|$ or $|left(\ell_2)\cap R|$ is at most $3j-\frac{t-1}{2}$ and the other is at least $3j-\frac{t-1}{2}$, then Lemma~\ref{lem:9} guarantees the existence of a directed line $\ell$ such that $|left(\ell)\cap (R\cup B)|=4j+\frac{t-1}{2}$ and $|left(\ell)\cap R|=3j-\frac{t-1}{2}$. Then, $|left(\ell)\cap B|=j+2\cdot\frac{t-1}{2}$, where $\frac{t-1}{2}$ is a nonnegative integer and $2j+1\geq\frac{5}{8}t$ implies that either $j\geq \frac{5}{8}\cdot\frac{t-1}{2}$ or $j\geq\lceil\frac{5}{8}\cdot\frac{t-1}{2}\rceil-1$ and $\frac{t-1}{2}\not=1,3$. By induction and Lemma~\ref{lem:11},
\[
{\mathcal{U}(left(\ell)\cap (R\cup B))}\leq \mathcal{U}\left(3j-\frac{t-1}{2},j+2\cdot\frac{t-1}{2}\right)\leq\frac{t-1}{2}.
\]
    {Also $|right(\ell)\cap R|=3(j+1)-\frac{t+1}{2}$ and $|right(\ell)\cap B|=(j+1)+2\cdot\frac{t+1}{2}$, where $\frac{t+1}{2}$ is a positive integer and $2j+1\geq\frac{5}{8}t$ implies that  $j+1\geq\frac{5}{8}\cdot\frac{t+1}{2}$}. By induction,
    \[{\mathcal{U}(right(\ell)\cap (R\cup B))\leq \mathcal{U}\left(3(j+1)-\frac{t+1}{2},j+1+2\cdot\frac{t+1}{2}\right)\leq\frac{t+1}{2}.}\]

    Therefore $\mathcal{U}(R\cup B)\leq\frac{t-1}{2}+\frac{t+1}{2}=t$.

    {Notice that, if there is no direction of $\ell_1 $ such that one of $|left(\ell_1)\cap R|$ or $|left(\ell_2)\cap R|$ is at most $3j-\frac{t-1}{2}$ and the other is at least $3j-\frac{t-1}{2}$, then, by Lemma~\ref{lem:9}, necessarily either $|left(\ell_1)\cap R|> 3j-\frac{t-1}{2}$ and $|left(\ell_2)\cap R|> 3j-\frac{t-1}{2}$ for every direction of $\ell_1 $ or $|left(\ell_1)\cap R|< 3j-\frac{t-1}{2}$ and $|left(\ell_2)\cap R|< 3j-\frac{t-1}{2}$ for every direction of $\ell_1$.}
\item[(b)] If $|left(\ell_1)\cap R|\geq 3j+1-\frac{t-1}{2}$ and $|left(\ell_2)\cap R|\geq 3j+1-\frac{t-1}{2}$ for every direction of $\ell_1 $, then
\begin{multline*}
|right(\ell_1)\cap right(\ell_2)\cap R|\leq 3k-t-2\left(3j+1-\frac{t-1}{2}\right)\\
=6j+3-t-6j-2+(t-1)=0.
\end{multline*}

That is, all the points in $(right(\ell_1)\cap right(\ell_2))\cap (R\cup B)$ are blue for every direction of $\ell_1$. But this is not possible as Lemma~\ref{lem:10} guarantees the existence of a direction where a point in $right(\ell_1)\cap right(\ell_2)$ is red.

\item[(c)] Finally, suppose that $|left(\ell_1)\cap R|\leq 3j-1-\frac{t-1}{2}$ and $|left(\ell_2)\cap R|\leq 3j-1-\frac{t-1}{2}$ for every direction of $\ell_1$. Since $|B|=k+2t\geq {5}$ {because $t$ and $k$ are odd, $t\ge 1$ and $k\ge 3$}, by Lemma~\ref{lem:10} there exists a directed line $\ell$ through a blue point $x$, such that  $|left(\ell)\cap (R\cup B)|=4j+\frac{t-1}{2}$. In this case, start with $\ell_1$ in the same direction as $\ell$. Hence, at most $4$ of the $5$ points in $(right(\ell_1)\cap right(\ell_2))\cap (R\cup B)$ are red ($x$ is one of these $5$ points and $x$ is blue) and so

\[|left(\ell_1)\cap R|=|left(\ell_2)\cap R|=3j-1-\frac{t-1}{2}=3j-\frac{t+1}{2}.\]

Then, $(left(\ell)\cap (R\cup B))\cup\{x\}$ is a $(3j-\frac{t+1}{2},j+2\cdot\frac{t+1}{2})$-set and $right(\ell)\cap (R\cup B)$ is a $(3(j+1)-\frac{t-1}{2},(j+1)+2\cdot\frac{t-1}{2})$-set, where $\frac{t\pm 1}{2}$ are nonnegative integers and $2j+1\geq\frac{5}{8}t$ implies that $j+1\geq\frac{5}{8}\cdot\frac{t-1}{2}$ and either $j\geq\frac{5}{8}\cdot\frac{t+1}{2}$ or $j\geq\lceil\frac{5}{8}\cdot\frac{t+1}{2}\rceil-1$ and $\frac{t+1}{2}\not=1$ or $3$.
By induction and Lemma~\ref{lem:11},
\begin{multline*}
\mathcal{U}(R\cup B)\leq \mathcal{U}\left(3j-\frac{t+1}{2},j+2\cdot\frac{t+1}{2}\right)\\
+\mathcal{U}\left(3(j+1)-\frac{t-1}{2},(j+1)+2\cdot\frac{t-1}{2}\right)\leq\frac{t+1}{2}+\frac{t-1}{2}=t.
\end{multline*}
\end{itemize}

\end{enumerate}

\subsection{Lower bound when $4/5 \le \alpha \le 1$}\label{sect:upperrange}

Now, let us prove the lower bound $\frac{8}{9}(r+b-8)$ on the number of points that can be always $K_{1,3}$-covered, when $4/5 \le \alpha \leq 1$. To this end, we use the following result.

\begin{theorem} [Kaneko, Kano and Suzuki~\cite{KKS2004}]\label{th:5}
Let $s\ge 1$, $g\ge 0$ and $h\ge 0$ be integers such that $g+h\ge 1$. Assume that $|R|=(s+1)g+sh$ and $|B|={sg}+(s+1)h$. Then, there exists a subdivision $X_1\cup\cdots\cup X_g\cup Y_1\cup\cdots\cup Y_h$ of the plane into $g+h$ disjoint convex regions such that every $X_i$ contains exactly $s+1$ red points and $s$ blue points and every $Y_j$ contains exactly $s$ red points and $s+1$ blue points.
\end{theorem}

{In order to obtain the upper bound, we now remove some red and blue points to build a point set with $5g+4h$ red points and $4g+5h$ blue points, for some values $g$ and $h$. Then, applying the previous theorem, we divide the plane into convex regions such that each of them contains either 5 red points and 4 blue points or 4 red points and 5 blue points. Using Theorem~\ref{th:10}, in each region 8 of the 9 points can be covered.}

\begin{theorem}\label{th:1}
Let $r$ and $b$ be positive integers and $\alpha = \frac{b}{r}$. If
$\frac{4}{5} \le \alpha \le 1$, then
\[
\mathcal{C}(r,b) \ge \frac{8}{9}(r+b-8)
\]
\end{theorem}

\begin{proof}
Assume $r+b \ge 9$, otherwise there is nothing to prove. Because
$\frac{4}{5} \le \alpha \le 1$, it follows that $4r \le 5b \le 5r$.
Write the nonnegative integer $5b-4r=9n+m$ for nonnegative integers
$n$ and $m$, $0 \le m \le 8$. Define
\[
(h,g) = \begin{cases} (n,r-b+n) &\mbox{if } 0 \le m \le 4 \\
(n+1,r-b+n) & \mbox{if } 5 \le m \le 8 \end{cases}
\]
Observe that if $0 \le m \le 4$, then $5h+4g=b-m$ and $4h+5g=r-m$;
and if $5 \le m \le 8$, then $5h+4g=b-(m-5)$ and $4h+5g=r-(m-4)$. In
either case, after removing $m$ blue points and $m$ red points (in
the first case) or $m-5$ blue points and $m-4$ red points (in the
second case), we end up with a set having $5h+4g$ blue points and
$4h+5g$ red points. According to Theorem~\ref{th:5} with $s=4$ (note
that $h+g \ge 1$ since $r+b \ge 9$), the plane can be partitioned
into $g+h$ disjoint convex regions $X_1\cup\cdots\cup X_g\cup
Y_1\cup\cdots\cup Y_h$ such that every $X_i$ contains exactly $5$
red points and $4$ blue points, and every $Y_j$ contains exactly $4$
red points and $5$ blue points. Finally, by Theorem~\ref{th:10}
(with $k=2$ and $t=1$), eight points in every $X_i$ and $Y_j$ can be
$K_{1,3}$-covered. Thus $\mathcal{C}(r,b) \ge \mathcal{C}(4h+5g,
5h+4g) \ge 8(h+g) \ge \frac{8}{9}(r+b-8)$.
\end{proof}

\subsection{Computing a covering of at least $\frac{8}{9}(r+b-8)$ points}\label{sect:complexity}

In this section we show how to compute a covering of at least $\frac{8}{9}(r+b-8)$ points for an $(r,b)$-set $S$ using an algorithm whose running time is $O(N^{\frac{4}{3}}\log^3(N))$, where $N=r+b$. {In particular, we prove the following theorem.}

\begin{theorem}\label{th:20}
Let $r\ge b$ be positive integers such that $\frac{1}{3}\leq \frac{b}{r} \leq 1$. If $r$ red points and $b$ blue points are given in the plane in general position, then at least $\frac{8}{9}(r+b)-4$ points can be $K_{1,3}$-covered in $O(N^{\frac{4}{3}}\log^3(N))$ time, where $N=r+b$.
\end{theorem}

{The proof of the theorem follows from the discussion in Subsections~\ref{subsub:first} and~\ref{subsub:second}. In these subsections, we study the complexity of finding such a covering when $1/3 \le \alpha \leq 4/5$ and when $4/5 \le \alpha \leq 1$, respectively.}

{\subsubsection{Computing a covering when $1/3 \le \alpha \leq 4/5$}\label{subsub:first}}

Let $1/3\le\alpha\leq 4/5$. Theorem~\ref{th:B-S} guarantees that at least $\frac{4}{7}\left(\frac{\alpha+2}{\alpha+1}\right)(r+b)-4$ points are covered, and this amount is at least $\frac{8}{9}(r+b-8)$ when $1/3 \le \alpha \leq 4/5$ (Corollary~\ref{cor:8_9_Lower}). The bound of the theorem is obtained by transforming the $(r,b)$-set into a $(3k-t,k+2t)$-set by removing $3\lceil \frac{s}{3}\rceil -s$ red points and $\lceil \frac{s}{3}\rceil$ blue points, where $k, t$ and $s$ are defined according to the formulas given in the proof of the theorem. Therefore, we only need to show how to cover with stars a $(3k-t,k+2t)$-set leaving uncovered at most $t$ points.

The proofs of Theorem~\ref{th:10} and Lemma~\ref{lem:11} provide methods to obtain such coverings for $(3k-t,k+2t)$-sets. The proof of Lemma~\ref{lem:11}, that corresponds to the particular case $k=\lceil\frac{5}{8}t\rceil-1$, is based on removing $i\le 8$ blue points to transform the $(r,b)$-set into a $\{1,3\}$-equitable set. By Theorem~\ref{th:3.4}, a $K_{1,3}$-covering can be found in $O(N^{\frac{4}{3}}\log^3(N))$ time.

In the proof of Theorem~\ref{th:10}, induction is applied in the four cases, dividing the original problem into two (disjoint) subproblems roughly involving half of the points each. Thus, a covering for a $(3k-t,k+2t)$-set can be found by building coverings for the corresponding subproblems. Note that, if $t=0$, then the $(3k-t,k+2t)$-set is a $\{1,3\}$-equitable set and, by Theorem~\ref{th:3.4}, a $K_{1,3}$-covering can be computed in $O(N^{\frac{4}{3}}\log^3(N))$ time.

Let $T(k,t)$ denote the complexity of finding a covering for a $(3k-t,k+2t)$-set, leaving uncovered at most $t$ points. According to Theorem~\ref{th:10}, \[T(k,t) \le T(k_1,t_1) + T(k_2,t_2) + O(N^{\frac{4}{3}}\log (N)) + O(N\log(N))\]
where $N=4k+t$. Summands $O(N^{\frac{4}{3}}\log (N))$ and $O(N\log(N))$ appear when Lemmas~\ref{lem:9} and~\ref{lem:10} are applied, respectively, and $k_1,k_2 \in \{\lfloor \frac{k}{2}\rfloor , \lceil \frac{k}{2}\rceil\}$ and $t_1,t_2 \in \{\frac{t}{2} - 1, \lfloor \frac{t}{2}\rfloor , \lceil \frac{t}{2}\rceil\}$. The exact values depend on the parity of $k$ and $t$. As $T(k,t) = O(N^{\frac{4}{3}}\log^3(N))$ when $t=0$ and $k=\lceil\frac{5}{8}t\rceil-1$, using standard techniques we can easily solve the recurrence, whose solution with equality is also $T(k,t) = O(N^{\frac{4}{3}}\log^3(N))$.

{\subsubsection{Computing a covering when $4/5 \le \alpha \leq 1$}\label{subsub:second}}

Now, let $4/5 \le \alpha \leq 1$. Given an $(r,b)$-set, by removing at most $4$ red points and $4$  blue points (as in the proof of  Theorem~\ref{th:1}), we obtain a $(4h+5g,4g+5h)$-set, for some integers $h$ and $g$. Applying Theorem~\ref{th:5}, we obtain $g+h$ convex regions containing 5 points of one color and 4 points of the other color each. Note that the 9 points in one these regions form a $(3k-t,k+2t)$-set (or a $(k+2t,3k-t)$-set) with $k=2$ and $t=1$. Instead of applying Theorem~\ref{th:10} to cover 8 of the 9 points for each one of these $(3k-t,k+2t)$-sets, such particular case allows us to provide an alternative method.

Assume we are given 5 red points and 4 blue points. Take a blue point $q$ and add a new blue point $q'$ very close to $q$. By the Ham-Sandwich theorem, there is a bisector~$\ell$ passing through one red point $p$ and one blue point $q''$, which can be found in linear time~\cite{Lo1994}. Suppose that $q'' = q$ and that, without loss of generality, $q'$ is in $right(\ell)$. By deleting $q'$, the $4$ points in $right(\ell)\cup\{p\}$ can be $K_{1,3}$-covered, $3$ blue points and $1$ red point in $left(\ell)\cup\{q\}$ can be $K_{1,3}$-covered, and these two $K_{1,3}$-coverings have no crossings. Similar reasonings apply to find the two stars when $q'' = q'$ or $q''\ne q,q'$. As a consequence, building all stars covering at least $\frac{8}{9}(r+b-8)$ points only requires linear time, after computing a partition of the plane as described in Theorem~\ref{th:5}. Therefore, the overall complexity depends on the complexity of computing such a partition. {The rest of this subsection is devoted to compute this partition in $O(N^{\frac{4}{3}}\log^3(N))$ time.}

Kaneko et al.~\cite{KKS2004} proved the existence of the partition, but do not provide any algorithm to find it. Following their proof and using the results from Bespamyatnikh et al.~\cite{BKS2000} and from Brodal and Jacob~\cite{BJ02}, we show how to compute the partition in $O(N^{\frac{4}{3}}\log^3(N))$ time.

First of all, we recall some definitions and results given in~\cite{BKS2000}. Let $S=R\cup B$ be a bicolored point set. A \emph{2-cutting} is a partition of the plane by a line $\ell$ into two halfplanes. Given a line $\ell$ and two integers $r$ and $b$, a 2-cutting is \emph{equitable} (or $(r,b)$-equitable)  if $|left(\ell)\cap R| = r$ and $|left(\ell)\cap B| = b$. A \emph{3-cutting} is a partition of the plane into three convex wedges $W_1, W_2$ and $W_3$ by three rays with a common point called the apex of the 3-cutting.  Given six integers $r_1, r_2, r_3, b_1, b_2, b_3$ such that $r_1+r_2+r_3=|R|$ and $b_1+b_2+b_3=|B|$, a 3-cutting is \emph{equitable} (or $(r_1, r_2, r_3, b_1, b_2, b_3)$-equitable) if each wedge~$W_i$ contains exactly $r_i$ red points and $b_i$ blue points. An equitable 2-cutting can be seen as a special case of an equitable 3-cutting, taking $r_2=b_2=0$, $r_3=|R|-r_1$, and $b_3=|B|-b_1$.

Choose $|R|$ vertical lines $\ell_i$ such that, for each line $\ell_i$, $1\le i\le |R|$, $|left(\ell_i)\cap R| = i$. Given two integers $i$ and $j$, we define $\operatorname{sign}(i,j) = -$, if $|left(\ell_i)\cap B| < j$, and $\operatorname{sign}(i,j) = +$, if $|left(\ell_i)\cap B| > j$. {This function controls whether the number of blue points to the left of $\ell_i$ is less than $j$ or not, when the number of red points to the left of $\ell_i$ is precisely $i$.} Note that if $|left(\ell_i)\cap B| = j$, then $\ell_i$ defines an $(i,j)$-equitable 2-cutting. The following theorem, which characterizes the existence of an equitable 3-cutting, was originally proved in~\cite{BKS2000} under the assumption that $\frac{r_1}{b_1} = \frac{r_2}{b_2} = \frac{r_3}{b_3}$. However, as mentioned in~\cite{KK2003,KKS2004}, this condition can be removed without changing the arguments in the proof given in~\cite{BKS2000}. Moreover, their algorithm to build an equitable 3-cutting also works in this more general setting.

\begin{theorem}[\cite{BKS2000,BJ02}]\label{th:3-cut}
Let $r_1, r_2, r_3, b_1, b_2, b_3$ be positive integers such that $|R| = r_1 +r_2 +r_3$ and $|B| = b_1+b_2+b_3$. If $\operatorname{sign}(r_1,b_1) = \operatorname{sign}(r_2,b_2) = \operatorname{sign}(r_3,b_3)$, then there exists either a $(r_1, r_2, r_3, b_1, b_2, b_3)$-equitable 3-cutting or a $(r_i,b_i)$-equitable 2-cutting for some $i=1,2,3$. Moreover, an equitable 3-cutting (2-cutting) can be computed in $O(N^{\frac{4}{3}}\log^2(N))$ time, with $N=|R|+|B|$.
\end{theorem}

As in Theorem~\ref{th:C}, the running time of the Bespamyatnikh et al.~\cite{BKS2000} algorithm to find a 3-cutting can be improved from  $O(N^{\frac{4}{3}}\log^3(N))$ to $O(N^{\frac{4}{3}}\log^2(N))$ by using the Brodal and Jacob~\cite{BJ02} maintenance of a convex hull in $O(\log(N))$ time per update.

In addition, Theorem 9 in~\cite{BKS2000} shows how to compute in linear time values $r_1, r_2, r_3, b_1, b_2, b_3$ satisfying $\operatorname{sign}(r_1,b_1) = \operatorname{sign}(r_2,b_2) = \operatorname{sign}(r_3,b_3)$ (with the additional constraint that $\frac{r_1}{b_1} = \frac{r_2}{b_2} = \frac{r_3}{b_3}$) such that $r_i \le \lfloor \frac{2|R|}{3}\rfloor$ and $b_i \le \lfloor \frac{2|B|}{3}\rfloor$.

Now we can go back to Theorem~\ref{th:5}. Recall that given $|R|=(s+1)g+sh$ and $|B|={sg}+(s+1)h$, with $s\ge 1$, $g\ge 0$ and $h\ge 0$ integers such that $g+h\ge 1$, we are interested in finding a subdivision $X_1\cup\cdots\cup X_g\cup Y_1\cup\cdots\cup Y_h$ of the plane into $g+h$ disjoint convex regions such that every $X_i$ contains exactly $s+1$ red points and $s$ blue points and every $Y_j$ contains exactly $s$ red points and $s+1$ blue points. We use $\mathcal{P}$ to denote such a partition. The proof of Theorem~\ref{th:5} given in~\cite{KKS2004} is based on proving the existence of either an equitable 2-cutting, with $r=(s+1)g'+sh'$ and $b={sg}'+(s+1)h'$ for some integers $0\le g'\le g$ and $0\le h'\le h$ such that $g'+h'<g+h$, or an equitable 3-cutting such that each wedge $W_i$ contains exactly $(s+1)g_i+sh_i$ red points and $sg_i+(s+1)h_i$ blue points, with $g_1+g_2+g_3=g$ and $h_1+h_2+h_3=h$. Then, induction is applied to each of the subproblems defined by the cutting.

Therefore, to build a 2-cutting or a 3-cutting, we need to show how to find values $r_1, r_2, r_3, b_1, b_2, b_3$ such that $\operatorname{sign}(r_1,b_1) = \operatorname{sign}(r_2,b_2) = \operatorname{sign}(r_3,b_3)$ {(to guarantee that such a cutting exists as in Theorem~\ref{th:3-cut})}, with the additional constraint that $r_i = (s+1)g_i+sh_i$, $b_i = sg_i+(s+1)h_i$, $g_1+g_2+g_3=g$ and $h_1+h_2+h_3=h$ {(as required in the proof of Theorem~\ref{th:5} in~\cite{KKS2004})}. To this end, we define a new set of signs and prove Theorem~\ref{th:g-h}.

Observe first that if $h=1$, then by coloring a blue point $q$ in red, we have a set consisting of $(s+1)(g+1)$ red points and $s(g+1)$ blue points. Thus, by Theorem~\ref{th:C}, we can obtain an equitable subdivision in $O(N^{\frac{4}{3}}\log^2(N)\log(g))$ time such that each of the $g+1$ convex regions contains $s+1$ red points and $s$ blue points. Coloring $q$ again in blue, the desired partition $\mathcal{P}$ is obtained. Using a similar reasoning when $g=1$, we can assume that $g,h\ge 2$.

Given two integers $g'\le g$ and $h'\le h$, we define $\operatorname{sg}(g',h') = -$, if $|left(\ell_i)\cap B| < {sg}'+(s+1)h'$, $\operatorname{sg}(g',h') = 0$, if $|left(\ell_i)\cap B| = {sg}'+(s+1)h'$, and $\operatorname{sg}(g',h') = +$, if $|left(\ell_i)\cap B| > {sg}'+(s+1)h'$, where $i=(s+1)g'+sh'$. {Note that, if $\operatorname{sg}(g',h') = 0$, then~$\ell_i$ defines an equitable 2-cutting with $(s+1)g'+sh'$ red points and $sg'+(s+1)h'$ blue points to the left of $\ell_i$.} Henceforth, we assume that there is no equality unless otherwise stated. By exchanging the colors if necessary, we may assume $\operatorname{sg}(1,0)=-$. Given two integers $g'$ and $h'$, the amounts $(m+1)g'+mh'$ and $(m+1)(g'-1) + m(h'+1)$ differ only by one, implying the following trivial observation.

\begin{observation}\label{obs:1}
If $\operatorname{sg}(g',h') = -$ then $\operatorname{sg}(g'-1,h'+1) = -$.
\end{observation}

In particular, $\operatorname{sg}(0,1)=-$. Suppose now that $\operatorname{sg}(g',h')=\operatorname{sg}(g-g', h-h')=-$. If $\ell_i$ and $\ell_j$ are the two lines such that $|left(\ell_i)\cap R| = (m+1)g'+mh'$ and $|left(\ell_j)\cap R| = (m+1)(g-g')+m(h-h')$, respectively, then $|left(\ell_i)\cap B| < mg'+(m+1)h'$, $|right(\ell_j)\cap B| > mg'+(m+1)h'$ and $|right(\ell_j)\cap R| = (m+1)g'+mh'$. By Lemma~\ref{lem:9}, there is a line defining an $((m+1)g'+mh',mg'+(m+1)h')$-equitable 2-cutting, and the same happens if $\operatorname{sg}(g',h')=\operatorname{sg}(g-g', h-h')=+$. In the spirit of Theorem~9 in~\cite{BKS2000}, the following theorem shows how to find two couples or three couples with the same sign {(according to the definition of $\operatorname{sg}(g',h')$)}, bounding some of the sizes. {Therefore, a 2-cutting or a 3-cutting exists with sizes as required in the proof of Theorem~\ref{th:5}.}

\begin{theorem}\label{th:g-h}
For any sequence of signs $\operatorname{sg}(g',h')$, there exist two couples $(g_1,h_1), (g_3,h_3)$, with $g_1+g_3=g$ and $h_1+h_3=h$, or three couples $(g_1,h_1), (g_2,h_2), (g_3,h_3)$, with $g_1+g_2+g_3=g$ and $h_1+h_2+h_3=h$, such that they have the same sign and $g_i \le \lfloor \frac{2g}{3}\rfloor$ for any $i$ or $h_i \le \lfloor \frac{2h}{3}\rfloor$ for any $i$.
\end{theorem}

\begin{proof}
Depending on the parity of $g$ and $h$, we have
four cases. If $g$ and $h$ are even, then
$(g/2,h/2)$ and $(g/2,h/2)$ satisfy the
statement. We explain in detail the case of $g$
and $h$ being odd. The other two cases, $g$ even
and $h$ odd and vice versa, can be analyzed in a similar and simpler way.

Let $g=2i+1$ and $h=2j+1$ for some integers $i$
and $j$. As $g,h\ge 2$, then $i,j\ge 1$. If
$\operatorname{sg}(i+1,j)=\operatorname{sg}(i,j+1)$,
we are done. Thus, we may assume that
$\operatorname{sg}(i+1,j)$ and
$\operatorname{sg}(i,j+1)$ are different and,
without loss of generality $\operatorname{sg}(i+1,j)= -$ and
$\operatorname{sg}(i,j+1)=+$. If
$\operatorname{sg}(i,j)=-$, then $(0,1), (i,j)$
and $(i+1,j)$ satisfy the theorem, so we may
assume that $\operatorname{sg}(i,j)=+$. We
distinguish two cases: $\operatorname{sg}(1,j)=+$
or $\operatorname{sg}(1,j)=-$.

Suppose first that $\operatorname{sg}(1,j)=+$. If
$\operatorname{sg}(i,0)=+$, then we are done, as
$\operatorname{sg}(i,j+1)=+$. Therefore, we may
assume that $\operatorname{sg}(i,0)=-$. Let $k_1
< j$ be the integer such that
$\operatorname{sg}(i,k_1)=-$ and
$\operatorname{sg}(i,k') = +$ for all $k'$,
$k_1<k'\le j$. Note that
$\operatorname{sg}(i,0)=-$ and
$\operatorname{sg}(i,j)=+$, so $k_1$ must exist.
Suppose that $k_1\ge \lfloor \frac{h}{3}\rfloor =
k_2$. Then, either $(i,k_1+1),(g-i,h-k_1-1)$ or
$(0,1),(i,k_1),(g-i,h-k_1-1)$ satisfy the
theorem. Hence, we may suppose that $k_1 < k_2$
and $\operatorname{sg}(i,k')=+$ for $k_2\le k'\le
j$. Take the couples $(1,j_1), (i,j_2), (i,j_3)$,
where $1\le j_1 \le k_2$, $j_2 = \lfloor
\frac{h-j_1}{2}\rfloor$ and $j_3 = \lceil
\frac{h-j_1}{2}\rceil$. Observe that $k_2 \le
j_2,j_3 \le j$, so $\operatorname{sg}(i,j_2)=+$
and $\operatorname{sg}(i,j_3)=+$ for all possible
values of $j_2$ and $j_3$. If there exists a
value of $j_1$ such that
$\operatorname{sg}(1,j_1)=+$, then the theorem
holds. Therefore, we may assume that
$\operatorname{sg}(1,j_1)=-$ for $j_1=1, \ldots ,
k_2$. Again, if $j > k'_1 \ge k_2$ is an integer
(which necessarily exists as
$\operatorname{sg}(1,j)=+$) such that
$\operatorname{sg}(1,k'_1)=-$ and
$\operatorname{sg}(1,k'_1+1) = +$, then, either
$(1,k'_1+1),(g-1,h-k'_1-1)$ or
$(0,1),(1,k'_1),(g-1,h-k'_1-1)$ satisfy the statement.

Suppose now that $\operatorname{sg}(1,j)=-$. Let
$k_1 < i$ be the integer such that
$\operatorname{sg}(k_1,j)=-$ and
$\operatorname{sg}(k',j) = +$ for all~$k'$ with
$k_1<k'\le i$. Recall that
$\operatorname{sg}(i,j)=+$. Suppose that $k_1\ge
\lfloor \frac{g}{3}\rfloor = k_2$. Then, either
$(k_1+1,j),(g-k_1-1,h-j)$ or
$(1,0),(k_1,j),(g-k_1-1,h-j)$ satisfy the
statement. Hence, we may suppose that $k_1 < k_2$
and $\operatorname{sg}(k',j)=+$ for $k_2\le k'\le i$.

By Observation~\ref{obs:1},
$\operatorname{sg}(0,j+1)=\operatorname{sg}(1,j)
= -$. Thus, we can repeat the previous reasoning,
as $\operatorname{sg}(i,j+1)=+$.  Let $k'_1 < i$
be the integer such that
$\operatorname{sg}(k'_1,j+1)=-$ and
$\operatorname{sg}(k',j+1) = +$ for all $k'$,
$k'_1<k'\le i$. If $k'_1\ge k_2$, then either
$(k'_1+1,j+1),(g-k'_1-1,h-j-1)$ or
$(1,0),(k'_1,j+1),(g-k'_1-1,h-j-1)$ satisfy the
statement, so we may suppose that $k'_1 < k_2$
and $\operatorname{sg}(k',j+1)=+$ for $k_2\le k'\le i$.

Now, consider the couples $(i_1,0), (i_2,j),
(i_3,j+1)$, where $1\le i_1 \le k_2$, $i_2 =
\lfloor \frac{g-i_1}{2}\rfloor$ and $i_3 = \lceil
\frac{g-i_1}{2}\rceil$. Note that $k_2 \le
i_2,i_3 \le i$, so $\operatorname{sg}(i_2,j)=+$
and $\operatorname{sg}(i_3,j+1)=+$ for all
possible values of $i_2$ and $i_3$. If there
exists a value of $i_1$ such that
$\operatorname{sg}(i_1,0)=+$, then the statement
holds. Therefore, we may assume that
$\operatorname{sg}(i_1,0)=-$ for $i_1=1, \ldots ,
k_2$. Suppose that $\operatorname{sg}(i,0)=+$.
Let $k''_1 \ge k_2$ be the integer such that
$\operatorname{sg}(k''_1,0)=-$ and
$\operatorname{sg}(k''_1+1,j) = +$. Then, either
$(k''_1+1,0),(g-k''_1-1,h)$ or
$(1,0),(k''_1,0),(g-k''_1-1,h)$ satisfy the
statement. On the contrary, suppose that
$\operatorname{sg}(i,0)=-$. By
Observation~\ref{obs:1},
$\operatorname{sg}(i-1,1)=-$, so
$(i-1,1),(i+1,j),(1,j)$ satisfy the statement as
$\operatorname{sg}(i+1,j)=-$ and $\operatorname{sg}(1,j)=-$.
\end{proof}

Now, we briefly explain the recursive algorithm to build a partition $\mathcal{P}$: Using Theorem~\ref{th:g-h}, we find couples $(g_1,h_1), (g_3,h_3)$ or couples $(g_1,h_1), (g_2,h_2), (g_3,h_3)$ with the same sign, and using Lemma~\ref{lem:9} or Theorem~\ref{th:3-cut}, we compute an equitable 2-cutting or an equitable 3-cutting. For each of the disjoint regions defined by the cutting, we continue recursively. Presorting the points according to their $x$-coordinates, which takes $O(N\log(N))$ time, couples $(g_1,h_1), (g_3,h_3)$ or couples $(g_1,h_1), (g_2,h_2), (g_3,h_3)$ can be clearly computed in linear time. Building an equitable 2-cutting requires $O(N^{\frac{4}{3}}\log (N))$ time (Lemma~\ref{lem:9}), and an equitable 3-cutting $O(N^{\frac{4}{3}}\log^2(N))$ time (Theorem~\ref{th:3-cut}). By Theorem~\ref{th:g-h}, at most $O(\log(g)+\log(h))$ iterations are required so that each of the obtained subproblems has $g$ or $h$ equals~1. Thus, Lemma~\ref{lem:9} and Theorem~\ref{th:3-cut} are applied at most $O(\log(g)+\log(h))$ times. Moreover, as all subproblems defined by the cuttings are disjoint, and building a partition $\mathcal{P}$ when $g=1$ or $h=1$ takes $O(n_i^{\frac{4}{3}}\log^3(n_i))$ time for a subproblem of size $n_i$, we have proved the following theorem.

\begin{theorem}\label{th:19}
Let $s\ge 1$, $g\ge 0$ and $h\ge 0$ be integers such that $g+h\ge 1$. Assume that $|R|=(s+1)g+sh$ and $|B|={sg}+(s+1)h$. Then, a subdivision $X_1\cup\cdots\cup X_g\cup Y_1\cup\cdots\cup Y_h$ of the plane into $g+h$ disjoint convex regions such that every $X_i$ contains exactly $s+1$ red points and $s$ blue points and every $Y_j$ contains exactly $s$ red points and $s+1$ blue points can be computed in $O(N^{\frac{4}{3}}\log^3(N))$ time, where $N=|R|+|B|$.
\end{theorem}

We remark that, in order to apply Theorem~\ref{th:g-h}, we assumed $\operatorname{sg}(g',h')\ne 0$ for all $g',h'$. But recall that, if $\operatorname{sg}(g',h')= 0$ for some values $g',h'$, the corresponding vertical line defines an equitable 2-cutting. Therefore, if $\operatorname{sg}(g',h')= 0$ for some values $g',h'$ when applying Theorem~\ref{th:g-h}, then we have directly found a 2-cutting, and we do not need to apply either Lemma~\ref{lem:9} or Theorem~\ref{th:3-cut}.

\section{Concluding remarks and open problems}\label{sec-4}

We have considered the problem of covering a bi-colored point set $S=R\cup B$ by graphs $K_{1,3}$, called stars, with straight legs and different colors in each class of the bipartition. First, in Sections~\ref{subsec-2.4},~\ref{subsec-2.2}, and~\ref{subsec-2.3} we have shown that some sets can be fully covered, provided that $|R|=r$ and $|B|=b$ allow it, like $\{1,3\}$-equitable sets, linearly separable sets, and double chains. In Sections~\ref{subsec-2.1} and~\ref{subsec-2.5} we have shown that there are sets for which a full cover is not possible, even if $r$ and $b$ allow it, like all sets in convex position, but also some sets in general position. Finally, in Section~\ref{sec-3}, we have proved that 8/9 of the points can always be covered, if $r$ and $b$ allow it.
The following question remains open:

\begin{openproblem}
Given a red-blue point configuration with $b\le r\le 3b$, is it always possible to $K_{1,3}$-cover at least $r+b-o(r+b)$ points?
\end{openproblem}

As a collateral contribution, in Table~\ref{table:small_known_values} we provide, for some values of  $r$ and $b$, the exact or possible values of~$\mathcal{U}(r,b)$, which denotes the maximum over all point sets in general position with $|R|=r$ and $|B|=b$ of the minimum number of points uncovered.
\begin{table}[!htb]
\begin{center}
\scalebox{0.7}{
  \begin{tabular}{ c|cccccccccccccccccccc}
\diagbox[width=2em]{$ r $}{ $ b $}
             & 1 & 2 & 3 & 4 & 5 & 6 & 7 & 8& 9 & 10 & 11 & 12 & 13 & 14 & 15 & 16  & 17 & 18 & 19 & 20\\ \hline

          1 & 2 & \multicolumn{19}{|c}{}\\ \cline{3-3}

          2 & 3 & \cellcolor{green!40}4 &\multicolumn{18}{|c}{}\\ \cline{4-4}

          3 & \cellcolor {cyan} 0  & \cellcolor {cyan} 1 & 2 &\multicolumn{17}{|c}{}\\ \cline{5-5}

          4 & \multicolumn{1}{c|}1 & 2 & \multicolumn{1}{c|}3 & \cellcolor{green!40}4 &\multicolumn{16}{|c}{}\\ \cline{6-6}

          5 & \multicolumn{1}{c|}2 & 3 & \cellcolor{green!40}4 & \cellcolor {cyan} 1 & 2 &\multicolumn{15}{|c}{}\\ \cline{7-7}

          6 & \multicolumn{1}{c|}3 & \cellcolor {cyan} 0 & 1 & \multicolumn{1}{c|}2 & \multicolumn{1}{c|}3 & \cellcolor{green!40}4 &\multicolumn{14}{|c}{}\\ \cline{8-8}

          7
           & 4 & \multicolumn{1}{c|}1 & 2 & \multicolumn{1}{c|} 3 & \cellcolor{green!40}4 & \cellcolor{green!40}5 & \cellcolor {cyan} 2 &\multicolumn{13}{|c}{}\\ \cline{9-9}

          8
           & 5 & \multicolumn{1}{c|}2 & 3 & \cellcolor{green!40}4 & \cellcolor {cyan} 1 & 2 & \multicolumn{1}{c|}3 & \cellcolor{green!40}4 &\multicolumn{12}{|c}{use the  symmetry}\\ \cline{10-10}

          9
          & 6 & \multicolumn{1}{c|}3 & \cellcolor {cyan} 0 & 1 & \multicolumn{1}{c|} 2 & 3 & \cellcolor{green!40}4 & \cellcolor{green!40}5 & \cellcolor {yellow}2 &\multicolumn{11}{|c}{}\\ \cline{11-11}

        10
        & 7 & 4 & \multicolumn{1}{c|}1 & 2 & \multicolumn{1}{c|}3
        & \cellcolor{green!40}4 & 1,5 & \cellcolor {cyan}2 & 3 & \cellcolor{green!40}4
        &\multicolumn{10}{|c}{}\\ \cline{12-12}

        11 & 8 & 5 & \multicolumn{1}{r|}2 & 3 & \cellcolor{green!40}4 & \cellcolor {cyan} 1 & 2 & \multicolumn{1}{c|}3 & \cellcolor{green!40}4 & \cellcolor{green!40}5
        & \cellcolor{orange}2 &\multicolumn{9}{|c}{}\\ \cline{13-13}

        12 & 9 & 6 & \multicolumn{1}{r|}3 & \cellcolor {cyan} 0 & 1 & \multicolumn{1}{r|}2 & 3 & \cellcolor{green!40}4 & 1,5 & 2,6 & \cellcolor{cyan}3 & \cellcolor{green!40}4 &\multicolumn{8}{|c}{}\\ \cline{14-14}

        13
        & 10 & 7 & 4 & \multicolumn{1}{r|}1 & 2
        & \multicolumn{1}{r|}3 & \cellcolor{green!40}4 & 1,5 & \cellcolor {cyan} 2 & 3
        & \cellcolor{green!40}4 & \cellcolor{green!40}5 & 2,6 &\multicolumn{7}{|c}{}\\ \cline{15-15}

        14
        & 11 & 8 & 5 & \multicolumn{1}{r|}2 & 3
        & \cellcolor{green!40}4 & \cellcolor {cyan} 1 & 2 & \multicolumn{1}{c|}3 & \cellcolor{green!40}4
        & 1,5 & 2,6 & \cellcolor {yellow}3 & \cellcolor{green!40}4 &\multicolumn{6}{|c}{}\\ \cline{16-16}

        15 & 12 & 9 & 6 & \multicolumn{1}{r|}3 & \cellcolor {cyan} 0 & 1 & \multicolumn{1}{r|}2 & 3 & \cellcolor{green!40}4 & 1,5 & 2,6 & \cellcolor {cyan}3 & \cellcolor{green!40}4 & \cellcolor{green!40}5 & 2,6 &\multicolumn{5}{|c}{}\\ \cline{17-17}

        16
        & 13 & 10 & 7 & 4 & \multicolumn{1}{r|}1
        & 2 & \multicolumn{1}{r|}3 & \cellcolor{green!40}4 & 1,5 & \cellcolor {cyan}2
        & 3 & \cellcolor{green!40} 4 & 1,5 & 2,6 & 3,7
        & 4,8&\multicolumn{4}{|c}{}\\ \cline{18-18}

        17
        & 14 & 11 & 8 & 5 & \multicolumn{1}{r|}2
        & 3 & \cellcolor{green!40}4 & \cellcolor {cyan} 1 & 2 & \multicolumn{1}{c|}3
        & \cellcolor{green!40}4 & 1,5 & 2,6 & 3,7 & \cellcolor{green!40}4 & \cellcolor{green!40}5  & 2,6 &\multicolumn{3}{|c}{}\\ \cline{19-19}

        18
        & 15 & 12 & 9 & 6 & \multicolumn{1}{c|}3
        & \cellcolor {cyan} 0 & 1 & \multicolumn{1}{c|}2 & 3 & \cellcolor{green!40}4
        & 1,5 & 2,6 & \cellcolor {cyan}3 & \cellcolor{green!40}4 & 1,5
        &  2,6 & 3,7& \cellcolor{yellow}4 &\multicolumn{2}{|c}{}\\ \cline{20-20}

        19
        & 16 & 13 & 10 & 7 & 4
        & \multicolumn{1}{r|}1 & 2 & \multicolumn{1}{c|}3 & \cellcolor{green!40}4 & 1,5
        & \cellcolor {cyan} 2 & 3 & \cellcolor{green!40}4 & 1,5 & 2,6
        & 3,7 & \cellcolor{yellow}4 & \cellcolor{green!40}5 & 2,6 &\multicolumn{1}{|c}{}\\ \cline{21-21}

        20
        & 17 & 14 & 11 & 8 & 5
        & \multicolumn{1}{r|}2 & 3 & \cellcolor{green!40}4 & \cellcolor {cyan} 1 & 2
        & \multicolumn{1}{c|}3  & \cellcolor{green!40}4 & 1,5 & 2,6 & 3,7
        & \cellcolor{green!40}4 & 1,5 & 2,6 & 3,7 & 4,8\\

%
%
%
%
\end{tabular}
}
    \caption{Exact values (or possible values) of $\mathcal{U}(r,b)$ for small $r$ and $b$.}
  \label{table:small_known_values}
\end{center}
\end{table}

First note that, by Theorem \ref{lem:PerfectCoverings}, all entries except those of the form $(r,3r)$ or $(3b,b)$ are positive. Also, clearly $\mathcal{U}(r,b) \equiv r+b \pmod{4}$. The blue cells come from Theorem~\ref{th:10}. The yellow cells are an application of Theorem~\ref{th:5} in the same way that Theorem~\ref{th:1} was proved. For instance, if $(r,b)=(19,17)$, then because $19=3\cdot5+4$ and $17=3\cdot4+5$; it follows by Theorem~\ref{th:5} that the plane can be partitioned into 4 convex regions such that three of them have 5 red points and 4 blue points, and the other has 4 red points and 5 blue points. By Theorem~\ref{th:10}, $\mathcal{U}(5,4)=1$, thus in each region 8 of the 9 points can be covered. Thus $\mathcal{U}(19,17) \le 4$, and by our previous remarks, $\mathcal{U}(19,17) = 4$. The other yellow cells are proved similarly. The white cells use the fact that
\begin{equation}\label{eq:possible_change}
\mathcal{U}(r,b+1)=
\left\{\begin{array}{l}
\mathcal{U}(r,b)+1\\
\text{ or }\\ \mathcal{U}(r,b)-3
\end{array}
\right.
\end{equation}
The green cells come from (\ref{eq:possible_change}), and from Theorem \ref{lem:PerfectCoverings} in the case of the entries that are 4 or 5. Finally, the orange cell is justified by the next result.
\begin{lemma}
$\mathcal{U}(11,11)=2$.
\end{lemma}

\begin{proof}
Consider two parallel lines $\ell_1$ and $\ell_2$ with opposite directions such that $|left(\ell_1) \cap (R \cup B)|=|left(\ell_2) \cap (R \cup B)|=9$.
First, suppose that $|left(\ell_1) \cap B| \ge 6$. It follows that $| left(\ell_2) \cap B| \le 5$ and so, by Lemma \ref{lem:9}, there is a line $\ell_3$ such that $|left(\ell_3) \cap (R \cup B)|=9$ and $|left(\ell_3) \cap B| = 6$. Then $left(\ell_3)\cap(R \cup B)$ is a $(3,6)$-set and $right(\ell_3)\cap(R \cup B)$ is an $(8,5)$-set. By Theorems~\ref{th:C} and~\ref{th:10}, $\mathcal{U}(3,6)=\mathcal{U}(8,5)=1$. Thus, it is possible to $K_{1,3}$-cover all but two points in $R \cup B$. The case $|left(\ell_2) \cap B| \ge 6$ is symmetric and, by exchanging the roles of~$R$ and~$B$, the previous reasoning holds as well for $|left(\ell_1) \cap R| \ge 6$ or $|left(\ell_2) \cap R| \ge 6$.

Hence, the following cases remain. First, if $|left(\ell_1) \cap B|=|left(\ell_2) \cap B| = 5$, then both $left(\ell_1) \cap (R \cup B)$ and $left(\ell_2) \cap (R \cup B)$ are $(4,5)$-sets, and $(R \cap B)\cap(right(\ell_1))\cap(right(\ell_2))$ is a $(3,1)$-set. Again by Theorem \ref{th:10}, $\mathcal{U}(4,5)=1$, so all but 2 points in $R\cup B$ can be $K_{1,3}$-covered. The symmetric situation exchanging the roles of $R$ and $B$ is proved analogously.
Second, if $|left(\ell_1) \cap B|=4$ and $|left(\ell_2) \cap B|=5$ then, by Lemma~\ref{lem:9}~(ii), there is a line $\ell_4$ through  a blue point such that $\ell_4$ leaves exactly 4 blue points and 4 red points to its left. As in the proof of Lemma~\ref{lem:9}, it can be assumed that the line $\ell_5$ in the opposite direction of~$\ell_4$ that has 9 points of $R \cup B$ to its left satisfies that $|left(\ell_5)\cap B|=5$  just as~$\ell_2$. Thus, in this direction we have exactly the situation of the previous case. Therefore $\mathcal{U}(11,11)=2$.
The symmetric situation exchanging the roles of~$\ell_1$ and~$\ell_2$ is proved analogously.
\end{proof}

Let us finish the present work with the remark that the proofs of Theorems~\ref{th:3.4} and~\ref{th:3.2}, for equitable sets and for linearly separable sets, can be extended to obtain, given a $s\ge 4$, $K_{1,s}$-coverings of $(sg+h,sh+g)$-sets, where $g$ and $h$ are non-negative integers. Thus, an interesting problem is that of studying the number of points that can be $K_{1,s}$-covered for an $(r,b)$-set.

\section{Acknowledgments}

Some parts of this work were done during the 2nd Austrian-Japanese-Mexican-Spanish Workshop on Discrete Geometry, held at the Universitat Polit\`ecnica de Catalunya during July 6-10, 2015. The authors would like to thank all participants of this workshop for their valuable discussions, as well as the anonymous referees for valuable comments which helped to improve the paper. 

\bibliographystyle{abbrv}

\end{document}